\documentclass[psamsfonts]{amsart}

\usepackage{amssymb,amsfonts}
\usepackage[all,arc]{xy}
\usepackage{enumerate}
\usepackage{mathrsfs}
\usepackage[utf8]{inputenc}
\usepackage{hyperref}
\usepackage{mathtools}
\usepackage{tikz}
\usepackage{enumitem}
\newtheorem{thm}{Theorem}[section]
\newtheorem{theorem}{Theorem}[section]
\newtheorem{cor}[thm]{Corollary}
\newtheorem{prop}[thm]{Proposition}

\newtheorem{lemma}[thm]{Lemma}

\newtheorem{question}[thm]{Question}
\newtheorem{claim}[thm]{Claim}
\newtheorem{fact}[thm]{Fact}

\theoremstyle{definition}
\newtheorem{definition}[thm]{Definition}

\theoremstyle{remark}
\newtheorem{remark}[thm]{Remark}

\newtheorem{observation}[thm]{Observation}
\newtheorem{notation}[thm]{Notation}

\makeatletter
\let\c@equation\c@thm
\makeatother
\numberwithin{equation}{section}

\title[Rado's Conjecture and its Baire version]{Rado's Conjecture and its Baire version}

\author{Jing Zhang}

\newcommand{\Addresses}{{
  \bigskip
  \footnotesize

 \textsc{Department of Mathematical Sciences,\\ Carnegie Mellon University, \\Pittsburgh, Pennsylvania, 15213}\par\nopagebreak
  \textit{E-mail}: \texttt{jingzhang@cmu.edu}
}}

\begin{document}

\begin{abstract}
Rado's Conjecture is a compactness/reflection principle that says any nonspecial tree of height $\omega_1$ has a nonspecial subtree of size $\aleph_1$. Though incompatible with Martin's Axiom, Rado's Conjecture turns out to have many interesting consequences that are also implied by certain forcing axioms. In this paper, we obtain consistency results concerning Rado's Conjecture and its Baire version. In particular, we show that a fragment of $\mathrm{PFA}$, which is the forcing axiom for \emph{Baire Indestructibly Proper forcings}, is compatible with the Baire Rado's Conjecture. As a corollary, the Baire Rado's Conjecture does not imply Rado's Conjecture.
Then we discuss the strength and limitations of the Baire Rado's Conjecture regarding its interaction with stationary reflection principles and some families of weak square principles. Finally we investigate the influence of Rado's Conjecture on some polarized partition relations.
\end{abstract}

\maketitle
\tableofcontents
\let\thefootnote\relax
\section{Introduction}\label{introduction}

\begin{definition}
A partial order $(T, <_T)$ is a \emph{tree} if for each $t\in T$, $\{s\in T: s<_T t\}$ is well ordered under the tree order $<_T$.
\end{definition}

When it is clear from the context, we will use $T$ to refer to $(T,<_T)$ and $<$ to refer to $<_T$.

\begin{definition}
For a given tree $T$, for each $t\in T$, the \emph{height} of $t$ in $T$ is the order type of its predecessors under the tree order, denoted as $ht_T(t)$. The height of the tree $T$ is the least ordinal $\alpha$ such that for all $t\in T$, $ht_T(t)<\alpha$.
\end{definition}

\begin{remark}
A tree $T$ 
\begin{enumerate}
\item is \emph{non-trivial} if each $t\in T$ has two incompatible extensions;
\item \emph{does not split on the limit levels} if for each limit $\alpha$ and $s,s'\in T$ such that $ht_T(s)=ht_T(s')=\alpha$, if $\{t\in T: t<s\}=\{t\in T: t<s'\}$, then $s=s'$.
\end{enumerate}
Restricting ourselves to trees that are non-trivial and do not split on the limit levels does not affect any application of Rado's Conjecture.
\end{remark}

The trees we deal with for the rest of the paper are non-trivial, are of height $\omega_1$, have a unique minimal element and do not split on the limit levels, unless otherwise stated. The unique minimal element (or the \emph{root}) of a tree $T$ will be denoted as $\mathrm{root}_T$.

Todor\v{c}evi\'{c} studied Rado's Conjecture, established some of its equivalent forms, and showed its consistency by collapsing a supercompact cardinal to $\omega_2$ in \cite{MR686495}. We will state Rado's Conjecture in terms of its tree formulation in this paper.

\begin{definition}
A tree $T$ is \emph{special} if there exists $g: T\to \omega$ such that $g$ is injective on chains.
\end{definition}

\begin{definition}\label{RC}
Rado's Conjecture ($\mathrm{RC}$) abbreviates the following: any nonspecial tree has a nonspecial subtree of size $\aleph_1$.
\end{definition}

Rado's Conjecture has interesting consequences. To sample a few:

\begin{theorem}[Todor\v{c}evi\'{c} \cite{MR1261218}, \cite{MR686495}] Rado's Conjecture implies:
\begin{enumerate}
\item $\theta^\omega = \theta$ for all regular $\theta\geq \omega_2$,
\item the Singular Cardinal Hypothesis,
\item $\square(\kappa)$ fails for all regular $\kappa\geq \omega_2$,
\item the Strong Chang's Conjecture,
\item for any regular cardinal $\lambda\geq \omega_2$, any stationary subset of $\lambda\cap \mathrm{cof}(\omega)$ reflects.
\end{enumerate}
\end{theorem}

\begin{theorem}[Feng \cite{MR1683892}]
Rado's Conjecture implies the non-stationary ideal on $\omega_1$ is presaturated.
\end{theorem}

\begin{theorem}[Doebler \cite{MR3065118}]\label{Doebler}
Rado's Conjecture implies that all stationary set preserving forcings are semiproper.
\end{theorem}

\begin{theorem}[Torres-P\'erez and Wu \cite{MR3600760}]
Rado's Conjecture along with $\neg \mathrm{CH}$ implies $\omega_2$ has the strong tree property. Rado's Conjecture implies the failure of $\square(\lambda,\omega)$ for all regular $\lambda\geq \omega_2$.
\end{theorem}

As remarked in \cite{MR686495} at the end of Section 3, many known consequences of Rado's Conjecture follow from a weaker principle, the \emph{Baire version} of Rado's Conjecture. 

Let $T$ be a given tree. We can view $T$ as a forcing poset such that $t'$ is a stronger condition than $t$ iff $t<_T t'$. Therefore, we can talk about subsets of $T$ that are open or dense in the context of a forcing poset.

\begin{definition}
A non-trivial forcing poset $\mathbb{P}$ is \emph{$\omega$-distributive} if forcing with $\mathbb{P}$ does not add new $\omega$-sequences of ordinals.
\end{definition}

 Notice that if $\mathbb{P}$ is separative, then $\mathbb{P}$ is $\omega$-distributive if and only if for any countable collection of open dense sets $\{U_n\subset \mathbb{P}: n\in \omega\}$, $\bigcap_n U_n$ is dense.

\begin{definition}
A tree is \emph{Baire} if it is $\omega$-distributive as a forcing notion.
\end{definition}

\begin{remark}
It is not always the case that a tree is separative. However, we do have that for any tree $T$, the following are equivalent: 
\begin{enumerate}
\item forcing with $T$ adds no new countable sequences of ordinals;
\item forcing with $T$ adds no new functions from $\omega$ to $V$;
\item for any countable collection of dense open subsets $\{U_n: n\in \omega\}$ of $T$, $\bigcap_n U_n$ is dense in $T$.
\end{enumerate}
The implications $(1)\leftrightarrow (2), (3)\rightarrow (1)$ are standard (see for example the section on Distributivity in Chapter 15 of \cite{MR1940513}, page 228 or Lemma IV.6.9 and Excercise IV.6.10 in Kunen \cite{MR2905394}). To see $(2)\rightarrow (3)$, assume $\{U_n: n\in \omega\}$ is a collection of dense open sets. Our goal is to show $\bigcap_{n\in \omega} U_n$ is dense. Let $t\in T$ be fixed and $G\subset T$ be generic over $V$ that contains $t$. In $V[G]$, define $f: \omega\to G\subset V$ recursively as follows: $f(0)=t$. Given $f(i)$, find $t'\geq f(i)$ in $U_i\cap G$ such that there exists two incompatible immediate extensions of $t'$ in $T$. To see that this is possible, since $f(i)\in G$ and $U_i$ is dense above $f(i)$, there is $t^*\in U_i\cap G$ with $t^*\geq f(i)$. By the non-triviality of $T$, there are $s,s'\geq t^*$ that are incompatible. Let $t'\geq t^*$ be such that it has two incompatible immediate extensions and no $s$ with $t^*\leq s < t'$ has this property. The existence of such $t'$ follows from the fact that $T$ does not split on the limit levels. Note by the openness of $U_i$, $t'\in U_i$. Define $f(i+1)=t'$. Let $\dot{f}$ be a $T$-name for the function $f$ as defined. By the hypothesis, we can find $t_0\geq t$ and $g: \omega\to T$ such that $t_0\Vdash \dot{f}=g$. In particular, we have $g(i+1)\in U_i$ for all $i\in \omega$. If the range of $g$ is finite, we are done. Otherwise, since $t_0\Vdash ``g(i)\in \dot{G}$ for all $i\in \omega$'', it must be the case that $t_0\geq g(i)$ for all $i\in \omega$. To see this, suppose for the sake of contradiction that there is $i>0$ such that $t_0 < g(i)$. Fix some $j>i$ such that $g(j)>g(i)$. As $g(i)$ has two incompatible immediate successors, we can extend $t_0$ to a condition that is incompatible with $g(j)$, contradicting with the fact that $t_0\Vdash g(j)\in \dot{G}$.
\end{remark}

Notice that any Baire tree is nonspecial. To see this, given a function $g: T\to \omega$, $U_n=\{t\in T: \exists t'<_T t \ g(t')=n \vee \forall t''\geq_T t \ g(t'')\neq n\}$ is a dense open subset of $T$ for all $n\in \omega$. If $g$ is a specializing map, then $\bigcap_n U_n =\emptyset$, witnessing that $T$ is not Baire. Hence the following is a statement weaker than $\mathrm{RC}$.

\begin{definition}\label{baireRC}
$\mathrm{RC}^B$ abbreviates the following: any Baire tree has a nonspecial subtree of size $\aleph_1$.
\end{definition}

We can also formulate a slightly stronger principle:
\begin{definition}
$s\mathrm{RC}^B$ abbreviates the following: any Baire tree has a Baire subtree of size $\aleph_1$.
\end{definition}

\begin{definition}\label{BIP}
A poset $P$ is \emph{Baire Indestructibly Proper (BIP)} if $P$ is proper and for any Baire tree $T$, $\Vdash_T$ $\check{P}$ is proper.
\end{definition}

\begin{lemma}[Todor\v{c}evi\'{c} \cite{MR657114}]
$\mathrm{MA}_{\aleph_1}$ is incompatible with $\mathrm{RC}^B$. 
\end{lemma}
\begin{proof}
Let $S\subset \omega_1$ be stationary co-stationary. 
Consider the tree $T(S)$ defined by the following: $t\in T(S)$ iff $t$ is a closed bounded subset of $S$. The order in $T(S)$ is end extension. $T(S)$ is the standard poset for adding a club subset contained in $S$. By Theorem 23.8 in \cite{MR1940513}, $T(S)$ is Baire. To see that $\mathrm{MA}$ is incompatible with $\mathrm{RC}^B$, simply note that by a theorem of Baumgartner, Malitz and Reinhard \cite{MR0314621}, $\mathrm{MA}_{\aleph_1}$ implies any $\aleph_1$-size subtree of $T(S)$ is special while $\mathrm{RC}^B$ implies there exists a nonspecial subtree of $T(S)$ of size $\aleph_1$.
\end{proof}

One of the motivations of our work is to understand what fragments of the standard forcing axioms are compatible with $\mathrm{RC}^B$. A natural guess is that it should include the ``non-specializing'' fragment. Our main result, Theorem \ref{iterationSacks}, shows that it could even include some ``harmless'' specializing forcings.

The main result of this paper is: 

\begin{theorem}\label{iterationSacks}
Assume the existence of a supercompact cardinal. There exists a  forcing extension where $s\mathrm{RC}^B$ and $\mathrm{MA}_{\omega_1}(\mathrm{BIP})$ both hold.
\end{theorem}

Since $\mathrm{MA}_{\omega_1}(\mathrm{BIP})$ implies the failure of $\mathrm{RC}$ (see Lemma \ref{BIPimpliesNotRC}), we have

\begin{cor}\label{separate}
$s\mathrm{RC}^B$ in general does not imply $\mathrm{RC}$.
\end{cor}

We are also interested in the influence of Rado's Conjecture on singular cardinal combinatorics, stationary reflection principles and polarized partition relations. 

\begin{definition}
For ordinals $\alpha, \beta$, let $\{\alpha\}^\beta$ denote $\{A\subset \alpha: \mathrm{otp}(A)=\beta\}$.
\end{definition}

\begin{definition}
$\begin{pmatrix}
\alpha \\
\beta
\end{pmatrix} \to \begin{pmatrix}
\gamma \\
\delta
\end{pmatrix}^{1,1}_{\sigma}$ abbreviates: for any $f: \alpha\times \beta \to \sigma$, there exists $A\in \{\alpha\}^\gamma$ and $B\in \{\beta\}^\delta$, such that $f\restriction A\times B$ is constant.
\end{definition}

\begin{definition}
Fix a cardinal $\kappa$ and a set $X$.
$P_\kappa X$ is defined to be $\{x\subset X: |x|<\kappa\}$. 
A subset $C\subset P_\kappa X$ is a \emph{closed unbounded} (or \emph{club}) subset of $P_\kappa X$ if there exists a function $f: P_\omega X \to P_\kappa X$ such that $C\supset Cl_f=_{def} \{x\in P_\kappa X: \forall z\in P_\omega x \ f(z)\subset x\}$. A subset $S\subset P_\kappa X$ is \emph{stationary} if for any club subset $C\subset P_\kappa X$, $S\cap C\neq \emptyset$.
\end{definition}

For more information on generalized stationarity, see Chapter 8 of \cite{MR1940513} or Lemma 0 of \cite{MR924672}.

\begin{definition}[\cite{MR924672}]\label{WRP}
For any regular $\lambda\geq \omega_2$, the \emph{Weak Reflection Principle} at $\lambda$, or $\mathrm{WRP}(\lambda)$, refers to the following principle: for any stationary $S\subset P_{\omega_1}\lambda$, there exists $W\subset \lambda$ such that $|W|=\aleph_1$, $\omega_1\subset W$ and $S\cap P_{\omega_1}W$ is stationary in $P_{\omega_1}W$. We use $\mathrm{WRP}$ to denote the following global principle: for any regular $\lambda\geq \omega_2$, $\mathrm{WRP}(\lambda)$.
\end{definition}

The paper is organized as follows: 
\begin{itemize}

\item In Section \ref{answer}, we sketch the proof that $\mathrm{RC}$ holds in the classical Mitchell model for the tree property. This model witnesses that $\mathrm{RC}+\neg \mathrm{CH}$ does not imply $\omega_2$ has the super tree property.
\item In Section \ref{separation}, we present a model via a mixed-support iteration where $s\mathrm{RC}^B$ holds but $\mathrm{RC}$ fails.

\item In Section \ref{iterate}, we prove Theorem \ref{iterationSacks}.
\item In Section \ref{strength}, we present some streamlined  proofs of known consequences of $\mathrm{RC}^B$ and show that $\mathrm{RC}^B$ in general is compatible with failures of certain simultaneous stationary reflection principles and some versions of weak square principles.
\item In Section \ref{polarized}, we show 
$\begin{pmatrix}
\omega_2 \\
\omega_1
\end{pmatrix} \to \begin{pmatrix}
\omega \\
\omega
\end{pmatrix}^{1,1}_{\omega}$ and 
$\begin{pmatrix}
\omega_2 \\
\omega_1
\end{pmatrix} \to \begin{pmatrix}
k \\
\omega_1
\end{pmatrix}^{1,1}_{\omega}$ for any $k\in \omega$ hold under a weak consequence of $\mathrm{RC}$, while it is consistent that $\mathrm{RC}$ holds but $\begin{pmatrix}
\omega_2 \\
\omega_1
\end{pmatrix} \to \begin{pmatrix}
\omega \\
\omega_1
\end{pmatrix}^{1,1}_{\omega}$ fails.

\end{itemize}

We end the introduction by including a simple lemma characterizing forcings that preserve $\omega$-distributivity, which is a variant of the well-known Easton's Lemma in the context of forcing with products.

\begin{definition}
A poset $\mathbb{P}$ is \emph{countably capturing} if for any $p\in \mathbb{P}$ and any $\mathbb{P}$-name $\dot{\tau}$ of a countable sequence of ordinals, there exists a $\mathbb{P}$-name $\dot{\sigma}$ such that $|\dot{\sigma}|\leq \aleph_0$, and there is $q\leq p$ such that 
$q\Vdash_{\mathbb{P}} \dot{\tau}=\dot{\sigma}$. 
\end{definition}

\begin{remark}\label{niceForm}
Here we think of each given $\mathbb{P}$-name $\dot{\tau}$ as represented by a function $f_{\dot{\tau}}$ whose domain is $\omega$ such that for each $n\in \omega$, $f_{\dot{\tau}}(n)=\{(\alpha_p, p): p\in A_n\}$ where $A_n$ is some maximal antichain chain of $\mathbb{P}$ such that for each $p\in A_n$, $p\Vdash_{\mathbb{P}} \dot{\tau}(n)=\alpha_p$. 
By saying $|\dot{\sigma}|\leq \aleph_0$, we really mean there exist antichains $B_n \in [\mathbb{P}]^{\leq \aleph_0}$ for $n\in \omega$ such that $\dot{\sigma}$ is represented by the function $n\in \omega\mapsto \{(\alpha_p, p): p\in B_n\}$. More explicitly, we mean that for any generic $G\subset \mathbb{P}$ over $V$, $(\dot{\sigma})^G$ is a partial function on $\omega$ such that for each $n\in \omega$, if there exists $p\in G\cap B_n$ (necessarily unique), then $(\dot{\sigma})^G(n)=\alpha_p$; otherwise $(\dot{\sigma})^G(n)$ is undefined.
\end{remark}

We may assume all the names for a countable sequence of ordinals in the following are represented as described in Remark \ref{niceForm}.

\begin{remark}\label{remarkProper}
Any proper forcing is countably capturing. To see this, let $p\in \mathbb{P}$ and a $\mathbb{P}$-name for a countable sequence of ordinals $\dot{\tau}$ be given. Let $\lambda$ be a sufficiently large regular cardinal and let $M\prec H(\lambda)$ be countable and contain $\mathbb{P}, p, \dot{\tau}$. By properness, we can find $q\leq p$ that is $(M,\mathbb{P})$-generic. Let $\dot{\sigma}=\dot{\tau}\cap M$. Then $|\dot{\sigma}|\leq \aleph_0$ and $q\Vdash_{\mathbb{P}} \dot{\sigma}=\dot{\tau}$.
\end{remark}

\begin{lemma}\label{capturing}
Let $\mathbb{P}$ be countably capturing and $\mathbb{Q}$ be $\omega$-distributive. Then TFAE:
\begin{enumerate}
\item $\Vdash_{\mathbb{P}} \check{\mathbb{Q}}$ is $\omega$-distributive
\item $\Vdash_{\mathbb{Q}} \check{\mathbb{P}}$ is countably capturing.
\end{enumerate}
\end{lemma}
\begin{proof}
\begin{itemize}
\item  2) implies 1): Let $G\times H$ be generic for $\mathbb{P}\times \mathbb{Q}$ and let $\dot{\tau}$ be a $(\mathbb{P}\times \mathbb{Q})$-name of a countable sequence of ordinals. We need to show $(\dot{\tau})^{G\times H}$ is in $V[G]$.
Since $\Vdash_\mathbb{Q} \mathbb{P}$ is countably capturing, in $V[H]$ there exists a $\mathbb{P}$-name $\dot{\sigma}$ with $|\dot{\sigma}|\leq \aleph_0$ such that in $V[H][G]$, $(\dot{\tau})^{H\times G}=(\dot{\sigma})^G$. Since $\mathbb{Q}$ is $\omega$-distributive, $\dot{\sigma}\in V$ . But then $(\dot{\tau})^{H\times G}=(\dot{\sigma})^{G}\in V[G]$.
\item  1) implies 2): Let $H$ be $\mathbb{Q}$-generic, we need to show $\mathbb{P}$ is countably capturing in $V[H]$. Let $\dot{\tau}$ be a $(\mathbb{Q}\times \mathbb{P})$-name for a countable sequence of ordinals. We can canonically identify $(\dot{\tau})^H$ as a $\mathbb{P}$-name for a countable sequence of ordinals in $V[H]$. Let $p\in \mathbb{P}$ be given. Let $G\subset \mathbb{P}$ containing $p$ be generic over $V[H]$. In $V[H\times G]$, by the assumption and the Product Lemma (see for example Theorem V.1.2 in \cite{MR2905394}, page 315), we know $(\dot{\tau})^{H\times G}\in V[G]$. Hence there exists a $\mathbb{P}$-name $\dot{\sigma}\in V$ such that $(\dot{\sigma})^G=(\dot{\tau})^{H\times G}$ in $V[G\times H]$. Find $q\in G, q\leq_{\mathbb{P}} p$ that forces $\dot{\sigma}=(\dot{\tau})^H$ in $V[H]$. In $V[H]$, $q\Vdash_{\mathbb{P}} ``\dot{\sigma}$ is a countable sequence of ordinals''. As ``$q\Vdash_{\mathbb{P}} \dot{\sigma}$ is a countable sequence of ordinals'' is a $\Delta_0$-statement, it also holds in $V$. By the fact that $\mathbb{P}$ is countably capturing in $V$, we can find $q'\leq_{\mathbb{P}} q$ and a $\mathbb{P}$-name $\dot{\varphi}$ such that $|\dot{\varphi}|\leq \aleph_0$ and $q'\Vdash_{\mathbb{P}} \dot{\varphi}=\dot{\sigma}$. Again since ``$q'\Vdash_{\mathbb{P}} \dot{\varphi}=\dot{\sigma}$'' is a $\Delta_0$-statement, it also holds in $V[H]$. Finally in $V[H]$, we have found $q'\leq_{\mathbb{P}} p$ such that $q'\Vdash_{\mathbb{P}} \dot{\varphi}=\dot{\sigma}=(\dot{\tau})^{H}$ and $|\dot{\varphi}|\leq \aleph_0$.
\end{itemize}
\end{proof}

We record the Pressing Down Lemma for trees due to Todor\v{c}evi\'{c} for later use.
\begin{theorem}[Todor\v{c}evi\'{c}, \cite{MR657114}]\label{PressingDownTree}
Fix a nonspecial tree $T$ and a \emph{regressive} function $f: T\to T$, namely for all $t\in T\backslash \{\mathrm{root}_T\}$, we have $f(t)<_T t$. Then there exists a nonspecial subtree $T'\subset T$ and $s\in T$ such that $f$ takes constant value $s$ on $T'$.
\end{theorem}

\section{$\mathrm{RC}+\neg \mathrm{CH}$ does not imply the super tree property at $\omega_2$}\label{answer}

Fix cardinals $\kappa\leq \lambda$. Recall the following definitions (see \cite{MR2838054} for instance).

\begin{definition}
$\langle d_a: a\in P_\kappa \lambda\rangle$ is a \emph{$P_\kappa\lambda$-list} if $d_a\subset a$ for all $a\in P_\kappa\lambda$.
\end{definition}

\begin{definition}
A $P_\kappa\lambda$-list $\langle d_a: a\in P_\kappa \lambda\rangle$ is \emph{thin} if there exists a club $C\subset P_\kappa\lambda$ such that $|\{d_a\cap c: c\subset a\in P_\kappa\lambda\}|<\kappa$ for every $c\in C$.
\end{definition}

\begin{definition}
Given $P_\kappa\lambda$-list $D=\langle d_a: a\in P_\kappa \lambda\rangle$ and $d\subset \lambda$, we say $d$ is an \emph{ineffable branch} of $D$ if there exists a stationary subset $S\subset P_\kappa\lambda$ such that $d\cap a = d_a$ for all $a\in S$.
\end{definition}

\begin{definition}
Let $\mathrm{ITP}(\kappa,\lambda)$ abbreviate: for any thin $P_\kappa\lambda$-list $D$, there exists an ineffable branch of $D$.
\end{definition}

\begin{definition}
We say $\kappa$ has the \emph{super tree property} if for any $\lambda\geq \kappa$, $\mathrm{ITP}(\kappa,\lambda)$ holds.
\end{definition}

Magidor proved in \cite{MR0327518} that
$\kappa$ is supercompact iff $\kappa$ is inaccessible and has the super tree property.

We show in this section using a model due to Mitchell \cite{MR0313057} that $\mathrm{RC} + \neg \mathrm{CH}$ does not imply the super tree tree property at $\omega_2$, answering a question of Torres-P\'erez and Wu \cite{MR3600760}. We heavily rely on Viale and Wei\ss's analysis in \cite{MR2838054}.

\subsection{Proof that $\mathrm{RC}$ holds in the Mitchell's model}\label{MitchellRC1}

We give a proof that $\mathrm{RC}$ holds in Mitchell's classical model of the tree property \cite{MR0313057}. This is due to Todor\v{c}evi\'{c}, who in \cite{MR1261218} pointed out the model works. In particular, this shows that $\mathrm{RC}$ is compatible with $2^\omega=\omega_2$. For completeness, we supply a proof here.

Let $\kappa$ be a strongly compact cardinal.
Let $\mathbb{M}_\kappa$ denote the Mitchell forcing, which consists of pairs $(p,f)$ where $p\in \mathrm{Add}(\omega,\kappa)$ and $f$ is a countably supported function on $\kappa$ such that for each $\alpha\in \kappa$, $f(\alpha)$ is an $\mathrm{Add}(\omega,\alpha)$-name for an element in $\mathrm{Add}(\omega_1, 1)^{V^{\mathrm{Add}(\omega,\alpha)}}$. The order $\leq$ in $\mathbb{M}_\kappa$ is defined as $(q,g)\leq (p,f)$ iff $q\supset p$ and for each $\alpha\in \mathrm{supp}(f)$, $q\restriction \alpha \Vdash_{\mathrm{Add}(\omega,\alpha)} g(\alpha)\leq_{\mathrm{Add}(\omega_1,1)} f(\alpha)$.

Let $R$ be the term poset. More precisely, conditions in $R$ are countably supported functions $f$ with domain $\kappa$ such that for each $\alpha\in \mathrm{supp}(f)$, $f(\alpha)$ is an $\mathrm{Add}(\omega,\alpha)$-name for an element in $\mathrm{Add}(\omega_1,1)^{V^{\mathrm{Add}(\omega,\alpha)}}$. For $f,g\in R$, $f\leq_R g$ iff $\mathrm{supp}(f)\supset \mathrm{supp}(g)$ and for each $\alpha\in \mathrm{supp}(g)$, $\Vdash_{\mathrm{Add}(\omega,\alpha)} f(\alpha)\leq g(\alpha)$. Observe that $R$ is countably closed.

We list a few well-known properties of the Mitchell poset, whose proofs can be found in \cite{MR0313057}:

\begin{enumerate}
\item $\mathbb{M}_\kappa$ is $\kappa$-c.c;
\item $\mathbb{M}_\kappa$ is a projection of $\mathrm{Add}(\omega,\kappa)\times R$ which is proper, so in particular $\mathbb{M}_\kappa$ is proper; 
\item for each inaccessible $\delta<\kappa$, we can truncate $\mathbb{M}_\kappa$ at $\delta$ to get $\mathbb{M}_\delta$. For any $G\subset \mathbb{M}_\delta$ generic over $V$, in $V[G]$, $\mathbb{M}_\kappa/G$ is forcing equivalent to a projection of $\mathrm{Add}(\omega, \kappa)\times R^*$, where $R^*$ is a countably closed poset.
\end{enumerate}

We need the following two general facts regarding non-specializing forcings.

\begin{claim}[\cite{MR686495}, Lemma 12]\label{notspecializedclosed}
No countably closed forcing can specialize a nonspecial tree.
\end{claim}

\begin{claim}\label{notspecializedcohen}
No Cohen forcing can specialize a nonspecial tree.
\end{claim}

\begin{proof}
Let $T$ be a given nonspecial tree. Let $\lambda$ be a cardinal and $\mathrm{Add}(\omega,\lambda)$ be the Cohen forcing for adding $\lambda$ reals. Suppose for the sake of contradiction that there exists a name for a specializing function $\dot{g}: T\to \omega$. For each $t\in T$, find $p_t \in \mathrm{Add}(\omega, \lambda)$ such that $p_t \Vdash \dot{g}(t)=n_t$ for some $n_t\in \omega$. Let $W=_{def}\bigcup\{\mathrm{dom}(p_t): t\in T\}$ and since $|W|\leq |T|$, there exists an injection $l: W\to T$. For each $t\in T$, $p_t$ is naturally identified with $p'_t$ such that $\mathrm{dom}(p'_t)=l[\mathrm{dom}(p_t)]$ and for each $s\in \mathrm{dom}(p'_t)$, $p'_t(s)=p_t(l^{-1}(s))$. It is immediate that whenever $s, s'\in T$ are such that $p'_s \cup p'_{s'}$ is a function, $p_s\cup p_{s'}$ is also a function. Hence, without loss of generality, we may assume $p_t$ is a finite partial function from $T$ to $2$.

 Let $F_t = \mathrm{dom}(p_t)$ for all $t\in T$. Since $T$ is not special, by going to a nonspecial subtree if necessary, we can assume there exists $m\in \omega$ and $n\in \omega$ such that $n_t=m$ and $|F_t|=n$ for all $t\in T$. Fix some well-ordering $\vartriangleleft$ on $T$.

We shrink the trees in $n$ rounds. Let $T_0=T$. At stage $i+1$, define a regressive function on $T_i$ such that $t\in T_i$ is mapped to 
\begin{enumerate}
\item its immediate predecessor if it has one, 
\item otherwise, the $<_T$-least strictly smaller node $s$ such that the $i$-th element (in the given order $\vartriangleleft$) of $F_t$ is in $F_s$ if such $s$ exists, 
\item otherwise, the root. 
\end{enumerate}

Apply Theorem \ref{PressingDownTree} and let $T_{i+1}$ be a non-special subtree on which the function is a constant, say $s_{i+1}$. Then we have the following property: for $t<t'\in T_{i+1}$, if the $i$-th element in $F_{t'}$ is in $F_t$, then it is already in $F_{s_{i+1}}$. Let $T'=T_n$. For any $t'\in T'$, we know that all elements in $F_{t'}$ that are also in $F_s$ for some $s<t'$ with $s\in T'$ are already in $D=_{def}\bigcup_{i < n}F_{s_i}$. Thus $F_t\cap F_{t'}\subset F_{t'}\cap D$ for every $t,t'\in T$ with $t<t'$. As $2^{[D]^{<\omega}}$ is finite, we can further find a nonspecial subtree $T^*\subset T'$, $r\in [D]^{<\omega}$ and $h: r\to 2$ such that for all $t\in T^*$, $F_t\cap D=r$ and $p_t\restriction r = h$. 
This implies that for any $t<t'\in T^*$, $p_t$ and $p_{t'}$ are compatible. This contradicts the fact that $\Vdash \dot{g}: T\to \omega$ is a specializing function.

\end{proof}

\begin{proof}[Proof that $\mathrm{RC}$ holds in $V^{\mathbb{M}_\kappa}$]
Let $G\subset \mathbb{M}_\kappa$ be generic over $V$ and let $T\in V[G]$ be a nonspecial tree of size $\theta$. In $V[G]$, we have that $\kappa=\omega_2$.
We may assume  $T$ is of the form $(\theta, <_T)$. Let $\lambda>\theta$ be a sufficiently large regular cardinal and fix $j: V\to M$ witnessing $\kappa$ is $\lambda$-strongly compact. More precisely, the embedding $j$ satisfies that $\mathrm{crit}(j)=\kappa$, ${}^\kappa M\subset M$ and for any $X\subset M$ such that $|X|\leq \lambda$ there is $Y\in M$ with $X\subset Y$ and $M\models |Y|<j(\kappa)$. By the $\kappa$-c.c-ness of $\mathbb{M}_\kappa$, $j\restriction \mathbb{M}_\kappa = id$ is a complete embedding of $\mathbb{M}_\kappa$ into $j(\mathbb{M}_\kappa)$. Moreover, it is not hard to see that $j(\mathbb{M}_\kappa)\restriction \kappa = \mathbb{M}_\kappa$. Fix some $Y\in M$ such that $M\models |Y|<j(\kappa)$ and $j''\theta \subset Y$. Let $K$ be generic over $V[G]$ for $j(\mathbb{M}_\kappa)/G$. We can lift $j$ to an elementary embedding $j^+ :V[G] \to M[G*K]$. In $M[G*K]$, we know that $(Y\cap j^+(\theta), <_{j^+(T)})$ is a subtree of $j^+(T)$ of size $<j(\kappa)$. In $V[G*K]$, we can see that $(Y\cap j^+(\theta), <_{j^+(T)})$ contains $(j'' \theta, <_{j^+(T)})$. We will be done if we manage to show that $(Y\cap j^+(\theta), <_{j^+(T)})$ is a nonspecial subtree of $j^+(T)$ (in $V[G*K]$ hence also in $M[G*K]$ by the downward absoluteness) as we can then finish by applying the elementarity of $j^+$. Since $(j'' \theta, <_{j^+(T)})\simeq (T, <_T)$ in $V[G*K]$, it is sufficient to show $(T,<_T)$ remains nonspecial after forcing with $j(\mathbb{M}_\kappa)/G$ over $V[G]$.

By the properties listed before Claim \ref{notspecializedclosed}, we know $j(\mathbb{M}_\kappa)/G$ is a projection of $(\mathrm{Add}(\omega,j(\kappa))\times R^*)^{M[G]}$ where $R^*\in M[G]$ is a countably closed poset in $M[G]$. But $R^*$ is also countably closed in $V[G]$, since $V\models M^\kappa\subset M$ and $\mathbb{M}_\kappa$ is $\kappa$-c.c, which implies that $V[G]\models M[G]^\kappa\subset M[G]$ (see for example Proposition 8.4 in \cite{MR2768691}). To summarize, in $V[G]$, $j(\mathbb{M}_\kappa)/G$ is a projection of $\mathrm{Add}(\omega,j(\kappa))\times R^{*}$ where $R^*$ is countably closed in $V[G]$. Apply Claim \ref{notspecializedclosed} and then Claim \ref{notspecializedcohen}, we know that $(T,<_T)$ remains nonspecial in $V[G*K]$.
\end{proof}

\subsection{Putting it together}

\begin{definition}
A forcing poset $\mathbb{P}$ such that $\Vdash_\mathbb{P} ``\kappa$ is regular'' has 
\begin{enumerate}
\item the $\kappa$-covering property if for any generic $G\subset \mathbb{P}$ and any set of ordinals $A\in V[G]$ such that $V[G]\models |A|<\kappa$, there exists $B\in V$ such that $V\models |B|<\kappa$ and $V[G]\models A\subset B$;
\item the $\kappa$-approximation property if for any generic $G\subset \mathbb{P}$ and any set of ordinals $A\in V[G]$, if $A\cap a\in V$ for all $a\in V$ with $V\models |a|<\kappa$, then $A\in V$.
\end{enumerate}
\end{definition}

\begin{remark}
For any poset $\mathbb{P}$ and regular $\kappa$, if $\mathbb{P}$ is $\kappa$-c.c, then $\mathbb{P}$ has the $\kappa$-covering property.
\end{remark}

We will need the following lemmas.

\begin{lemma}[\cite{MR3072773}, Lemma 2.4]\label{unger}
For $\kappa$ regular, if a poset $\mathbb{P}$ satisfies that $\mathbb{P}\times \mathbb{P}$ has $\kappa$-c.c, then $\mathbb{P}$ has $\kappa$-approximation property.
\end{lemma}

\begin{lemma}[\cite{MR0313057}, Lemma 3.3]\label{Knaster}
$\mathbb{M}_\kappa$ is $\kappa$-Knaster. 
\end{lemma}

In fact, Lemma 3.3 of \cite{MR0313057} shows $\mathbb{M}_\kappa$ is $\kappa$-c.c. But essentially the same proof gives that $\mathbb{M}_\kappa$ is $\kappa$-Knaster. The key modification is that the $\kappa$-Knasterness of $\mathrm{Add}(\omega,\kappa)$ will be used in the proof instead of just the $\kappa$-c.c-ness of $\mathrm{Add}(\omega,\kappa)$. Alternatively, the proof can also be adapted from the proof of Claim \ref{newChain} in Section \ref{separation} with the aforementioned modification in mind.

In particular, by Lemma \ref{unger} and Lemma \ref{Knaster}, $\mathbb{M}_\kappa$ satisfies the $\kappa$-approximation property and the $\kappa$-covering property.

We use the following result due to Viale and Weiss \cite{MR2838054}. 

\begin{theorem}[\cite{MR2838054}]\label{VialeWeiss}
Let $\kappa$ be an inaccessible cardinal and $\mathbb{P}$ be a proper poset with the $\kappa$-covering and $\kappa$-approximation properties. If in $V^\mathbb{P}$, the super tree property holds at $\kappa$, then the super tree property must already hold in $V$ at $\kappa$. In particular, $\kappa$ must be supercompact in $V$.
\end{theorem}

\begin{theorem}\label{MitchellRC}
Let $\kappa$ be a strongly compact cardinal that is not supercompact. Then there exists a forcing extension in which $\mathrm{RC} $ and $2^\omega=\kappa=\omega_2$ hold but the super tree property at $\omega_2$ fails.
\end{theorem}

\begin{proof}

Force with $\mathbb{M}_\kappa$ and let $G\subset \mathbb{M}_\kappa$ be generic over $V$. In $V[G]$, $\mathrm{RC} + 2^\omega=\kappa=\omega_2$ hold by the discussion in Subsection \ref{MitchellRC1} but the super tree property at $\kappa=(\omega_2)^{V^{\mathbb{M}_\kappa}}$ fails by Theorem \ref{VialeWeiss} and Lemma \ref{Knaster}. 

\end{proof}

To end this section, we discuss how the Mitchell extension can serve as a simple model to separate $\mathrm{RC}$ from the stationary reflection principle $\mathrm{WRP}$ (recall Definition \ref{WRP}). More precisely, we show that $\mathrm{WRP}$ fails in the model constructed in Theorem \ref{MitchellRC}. 
To see this, let $G\subset \mathbb{M}_\kappa$ be generic over $V$. First notice:
\begin{claim}\label{Cohen}
$V[G] \models \mathrm{MA}_{\omega_1}(\mathrm{Cohen})$.
\end{claim}

\begin{proof}[Proof of the Claim]
Since $\kappa$ is measurable in $V$, we can fix an elementary embedding $j: V\to M$ definable in $V$ such that $\mathrm{crit}(j)=\kappa$ and ${}^\kappa M\subset M$. In $V[G]$, fix a collection of dense open subsets of $\mathrm{Add}(\omega,1)$, say $\bar{D}=\{D_\alpha\subset \mathrm{Add}(\omega,1): \alpha<\omega_1\}$. 

Since ${}^\kappa M \subset M$, we know that $\mathbb{M}_\kappa\in M$. It is not hard to see that in $M$, $\mathbb{M}_\kappa$ embeds completely into $j(\mathbb{M}_\kappa)$ (see Section 23 in \cite{MR2768691} for more details). 
Let $H \subset j(\mathbb{M}_\kappa)/G$ be generic over $V[G]$. By a theorem of Silver (see Proposition 9.1 in \cite{MR2768691} for a proof), $j$ lifts to $j^+: V[G]\to M[G*H]$. Since $\mathbb{M}_\kappa$ is $\kappa$-c.c by Theorem \ref{Knaster} and $V\models {}^\kappa M\subset M$, we know by Proposition 8.4 in \cite{MR2768691} that $V[G]\models {}^\kappa M[G]\subset M[G]$. Therefore, $\bar{D}\in M[G]$. Since $j(\mathbb{M}_\kappa)/G$ projects onto $\mathrm{Add}(\omega,1)$ in $M[G]$, we know there exists a filter $h\subset \mathrm{Add}(\omega,1)$ in $M[G*H]$ that is generic over $M[G]$. As a result, $h$ meets $\bar{D}$ in $M[G*H]$. Since $j^+(\bar{D})={j^+} '' \bar{D}=\bar{D}$, by the elementarity of $j^+$, we know $V[G] \models $ ``there exists a filter $h\subset \mathrm{Add}(\omega,1)$ that meets $\bar{D}$''. Thus, we have shown that $V[G] \models ``\mathrm{MA}_{\omega_1}(\mathrm{Cohen})$''.
\end{proof}

Suppose for the sake of contradiction that $\mathrm{WRP}$ holds in $V[G]$. By Theorem 3.1 in \cite{MR3284479} and Claim \ref{Cohen}, we have that $\omega_2$ has the super tree property in $V[G]$. This contradicts Theorem \ref{MitchellRC}.

\begin{remark}
Sakai (\cite{MR2387938}, Section 5) established the consistency of $\mathrm{SSR}+\neg \mathrm{WRP}(\omega_3)$ relative to the consistency of a supercompact cardinal. In fact, in his model, $\mathrm{RC}+\neg \mathrm{WRP}(\omega_3)$ holds. The forcing poset used there is different from the Mitchell poset we use here and the analysis of various intermediate forcing extensions is quite delicate.
The consistency of the existence of a strongly compact cardinal that is not supercompact is due to Menas \cite{MR0357121}. There are also ways of getting extreme examples, for instance it is possible for a strongly compact cardinal to be the least measurable cardinal (see \cite{MR0429566} or Proposition 22.6 in \cite{MR2768691}).
\end{remark}

\section{Separating $s\mathrm{RC}^B$ from $\mathrm{RC}$}\label{separation}

In this section we show $s\mathrm{RC}^B$ does not imply $\mathrm{RC}$ in general. Recall that we mentioned this result in the introduction as a corollary of Theorem \ref{iterationSacks}, which will be proved in Section \ref{iterate}.
However, since the model presented in this section has a certain ``minimal'' flavor and is different from the model constructed in Section \ref{iterate}, we believe it is of independent interest. 

We start off introducing a tree that will be central in the proof.

\begin{definition}
Let $\sigma(\mathbb{R})$ denote the tree consisting of bounded subsets of $\mathbb{R}$ well ordered by the natural order on $\mathbb{R}$. The order in $\sigma(\mathbb{R})$ is end extension.
\end{definition}

We list a few observations about $\sigma(\mathbb{R})$.

\begin{observation}\label{observation}
\begin{enumerate}
\item $\sigma(\mathbb{R})$ is nonspecial (Todor\v{c}evi\'{c} \cite{MR686495}, Example 7);
\item $\sigma(\mathbb{R})$ is not Baire (for each $n\in \omega$, $U_n=_{def} \{t\in \sigma(\mathbb{R}): \exists \alpha\in t \ \alpha>n\}$ is a dense open subset of $\sigma(\mathbb{R})$ but $\bigcap_{n\in \omega} U_n=\emptyset$);
\item For any subset $U\subset \sigma(\mathbb{R})$, in any outer model, $U$ has no uncountable branches.
\end{enumerate}
\end{observation}

Given a tree $T$, let $S(T)$ denote the specializing poset of $T$. More precisely, it contains finite functions $s: T\to \omega$ that are injective on chains. The poset is ordered by the reverse inclusion. We need the following characterization due to Baumgartner.

\begin{theorem}[\cite{MR823775}, \cite{MR0314621}]\label{Baumgartner}
$S(T)$ is c.c.c iff $T$ does not contain an uncountable branch.
\end{theorem}

Let $\kappa$ be a  supercompact cardinal. Let $\langle P_\alpha, \dot{Q}_\beta: \alpha\leq \kappa, \beta<\kappa\rangle$ be the finite support iteration of c.c.c forcings such that for each $\alpha$, $\Vdash_\alpha \dot{Q}_\alpha = S(\sigma(\mathbb{R})^{V^{P_\alpha}})$.
Observe that $\Vdash_\kappa$ all subsets of $\sigma(\mathbb{R})$ of size $<\kappa$ are special. The reason is that: $\mathbb{P}_\kappa$ is c.c.c and each $<\kappa$-subset of $\sigma(\mathbb{R})$ in $V^{\mathbb{P}_\kappa}$ appears in $V^{\mathbb{P}_\alpha}$ as a subset of $\sigma(\mathbb{R})^{V^{\mathbb{P}_\alpha}}$ for some $\alpha<\kappa$.

\begin{lemma}\label{indestructible}
For any Baire $T$, $T\Vdash P_\kappa$ is ccc. Hence, by Remark \ref{remarkProper} and Lemma \ref{capturing}, $P_\kappa\Vdash T$ is Baire.
\end{lemma}

\begin{proof}
We induct on $\alpha\leq \kappa$. When $\alpha$ is a limit ordinal, let $H\subset T$ be generic over $V$. Then in $V[H]$, $P_\alpha$ is the direct limit of $\langle P_\beta: \beta<\alpha\rangle$ and each $P_\beta$ is c.c.c by the induction hypothesis. So by the usual argument via the $\Delta$-System Lemma, we know that $P_\alpha$ is also c.c.c in $V[H]$. When $\alpha=\beta+1$, let $H\times G\subset T\times P_\beta$ be generic over $V$. We examine $Q_\beta=(\dot{Q}_\beta)^{G}$ in $V[H\times G]$. By induction hypothesis, $\Vdash_T \check{P}_\beta$ is c.c.c. By Lemma \ref{capturing}, $\Vdash_{P_{\beta}} T$ is Baire. By our definition of the iteration, $Q_\beta$ lives in $V[G]$, and is a specializing poset for $(\sigma(\mathbb{R}))^{V[G]}$. Note that $(\sigma(\mathbb{R}))^{V[G]}=(\sigma(\mathbb{R}))^{V[G\times H]}$ since $T$ is Baire in $V[G]$ and $(\sigma(\mathbb{R}))^{V[G]}$ does not have any uncountable branch in $V[G\times H]$. Therefore, $Q_\beta$ is the same as the Baumgartner specializing poset for $\sigma(\mathbb{R})$ as defined in $V[G\times H]$. By Theorem \ref{Baumgartner}, $Q_\beta$ is c.c.c. in $V[G\times H]$. 
\end{proof}

\begin{remark}
Lemma \ref{indestructible} remains valid if we replace the Baire tree $T$ with any $\omega$-distributive forcing $\mathbb{P}$.
\end{remark}

We define our main forcing as a variant of Mitchell's forcing for the tree property.

\begin{definition}
$Q$ is a poset consisting of $(p,f)$ where $p\in P_\kappa$ and $f$ is a countably supported function on $\kappa$ and for each $\alpha\in \mathrm{dom}(f)$, $f(\alpha)$ is a $P_\alpha$-name for a condition in $(\mathrm{Add}(\omega_1, 1))^{V^{P_\alpha}}$. The order in $Q$ is given by $(p_1,f_1)\leq (p_2, f_2)$ iff $p_1\leq_{P_\kappa} p_2$, $\mathrm{supp}(f_1)\supset \mathrm{supp}(f_2)$ and for each $\alpha\in \mathrm{supp}(f_2)$, $p_1\restriction \alpha \Vdash_{P_\alpha} f_1(\alpha)\leq f_2(\alpha)$.
\end{definition}

\begin{claim}\label{newChain}
$Q$ is $\kappa$-c.c.
\end{claim}

\begin{proof}
Let $\langle (p_\alpha,f_\alpha): \alpha<\kappa \rangle\subset Q$ be given.
Apply the $\Delta$-System Lemma (see Lemma III.6.15 in \cite{MR2905394} for a proof) to get $A\in [\kappa]^\kappa$ such that $\{\mathrm{dom}(f_\alpha): \alpha\in A\}$ forms a $\Delta$-system with root $h\in [\kappa]^{<\omega_1}$ and for all $\alpha,\beta\in A$, $f_\alpha\restriction h = f_\beta\restriction h$. This is possible since for any $\omega\leq \beta<\kappa$, the collection of nice $P_\beta$-names for $\mathrm{Add}(\omega_1, 1)$ is contained in $V_{(2^{|P_\beta|})^+}$ and $h$ is countable.

Since $P_\kappa$ is c.c.c, we may find $\alpha<\beta\in A$ such that $p_\alpha$ and $p_\beta$ are compatible. Fix some $r\leq p_\alpha, p_\beta$. Let $f_\alpha + f_\beta$ be the function defined as follows: 
\begin{equation*}
f_\alpha + f_\beta(\gamma)=\begin{cases}
f_\alpha(\gamma) & \gamma\in h,\\
f_\alpha(\gamma) & \gamma\in \mathrm{supp}(f_\alpha)-h,\\
f_\beta(\gamma) & \gamma\in \mathrm{supp}(f_\beta)-h,\\
\emptyset^{V^{\mathrm{Add}(\omega,\gamma)}} & \text{otherwise}.
\end{cases}
\end{equation*}
It can be easily checked that $(r, f_\alpha+f_\beta)$ is an element in $Q$ extending both $(p_\alpha,f_\alpha)$ and $(p_\beta, f_\beta)$.
\end{proof}

We will recall some standard analysis of this poset.
Let $R$ be the poset consisting of functions $f$ with domain $\kappa$ of countable support such that for each $\alpha\in \kappa$, $f(\alpha)$ is a $P_\alpha$-name for an element in $\mathrm{Add}(\omega_1, 1)^{V^{P_\alpha}}$ and for any $f, g\in R$, $f\leq_R g$ iff $\mathrm{supp}(f)\supset \mathrm{supp}(g)$ and for each $\beta\in \mathrm{supp}(g)$, $\Vdash_{P_\beta} f(\beta)\leq g(\beta)$. Notice $R$ is countably closed.

\begin{claim}\label{firstCoordinate}
$Q$ projects onto $P_\kappa$.
\end{claim}
\begin{proof}
The projection onto the first coordinate works.
\end{proof}

\begin{claim}\label{projection}
$P_\kappa\times R$ projects onto $Q$.
\end{claim}

\begin{proof}
Consider the identity map. To see that it is a projection map, for each $(p,f)\in P_\kappa\times R$, and $(q,g)\leq_Q (p,f)$, we need to find $(p',f')\in P_\kappa\times R$ such that $(p',f')\leq_{P_\kappa\times R} (p,f)$ and  $(p',f')\leq_Q (q,g)$. Let $p'=q$. Let $f'$ be a function with support $\mathrm{supp}(g)$ and for each $\beta\in \mathrm{supp}(g)$, $q\restriction \beta \Vdash_{P_\beta} f'(\beta)=g(\beta)$ and $\Vdash_{P_\beta} f'(\beta)\leq f(\beta)$. We can find such a function by the maximality principle of forcing.
\end{proof}

\begin{claim}\label{countablyClosed}
For any countably closed $\mathbb{D}$ and any Baire tree $T$, $\mathbb{D}\times T$ is $\omega$-distributive.
\end{claim}

\begin{proof}
It immediately follows from the fact that $\Vdash_T \mathbb{D}$ is countably closed, as $T$ is Baire.
\end{proof}

We need similar product analysis of the quotient forcing. Let $\delta<\kappa$ be inaccessible. We can truncate $Q$ to $Q\restriction \delta$ in the obvious way. Let $G_\delta$ be generic for $Q\restriction \delta$. Let $H_\delta$ be the projection of $G_\delta$ to the first coordinate, which is $V$-generic for $P_\delta$.

\begin{claim}\label{quotient}
Let $T\in V[G_\delta]$ be a Baire tree. Then in $V[G_\delta]$, $\Vdash_{Q/G_\delta} T$ is a Baire tree.
\end{claim}

\begin{proof}

In $V[G_\delta]$, we will show that, similarly as in Claim \ref{projection}, $Q/G_\delta$ is a projection of $(P_{[\delta, \kappa)})^{V[H_\delta]}\times R^*$, where $R^*$ is some countably closed poset in $V[G_\delta]$. Let $E=(P_{[\delta, \kappa)})^{V[H_\delta]}$. Let us give a more detailed description of what $R^*$ is and what the projection is. The elements in $R^*$  are countably supported functions $f$ with domain $[\delta,\kappa)$ such that for each $\beta\in \mathrm{dom}(f)$, $f(\beta)\in V[H_\delta]$ is a $(P_{[\delta,\beta)})^{V[H_\delta]}$-name for an element in $\mathrm{Add}(\omega_1,1)$.
In $V[G_\delta]$, $f\leq_{R^*}g$ if $\mathrm{supp}(f)\supset \mathrm{supp}(g)$ and for each $\gamma\in \mathrm{supp}(g)$, $\Vdash_{E} f(\gamma)\leq_{\mathrm{Add}(\omega_1,1)} g(\gamma)$.

To see that $R^*$ is countably closed in $V[G_\delta]$, it is sufficient to notice that the quotient forcing $D=(Q\restriction \delta)/H_\delta$ is $\omega$-distributive, which is due to our product analysis in Claim \ref{projection}. In particular, $V[G_\delta]\models {}^\omega V[H_\delta]\subset V[H_\delta]$, $\mathbb{R}^*\in V[H_\delta]$ and $R^*$ is countably closed in $V[G_\delta]$. 

Notice that in $V[G_\delta]$, $Q/G_\delta $ is forcing equivalent to the poset $B$ such that $(s,f)\in B$ iff $s\in E$ and $f$ is a countably supported function with domain $[\delta,\kappa)$ and $range(f)\subset V[H_\delta]$ such that for any $\alpha\in \mathrm{supp}(f)$, $\Vdash_{E\restriction \alpha} f(\alpha)\in \mathrm{Add}(\omega_1,1)$. The order in $B$ is given by $(s',f')\leq (s,f)$ iff $s'\leq_{E} s$, $\mathrm{supp}(f')\supset \mathrm{supp}(f)$ and for each $\alpha\in \mathrm{supp}(f)$, $s'\restriction \alpha\Vdash_{E\restriction \alpha} f'(\alpha)\leq_{\mathrm{Add}(\omega_1,1)} f(\alpha)$. By a similar argument as in Claim \ref{projection}, we check that $id: E\times R^* \to B$ is a projection in $V[G_\delta]$. To see this, for each $(s,f)\in E\times R^*$, and each $(s',f')\leq_B (s,f)$, we need to find $(s'',f'')\leq_{E\times R^*} (s,f)$ such that $(s'',f'')\leq_B (s',f')$. Let $s''=s'$ and $\mathrm{supp}(f'')=\mathrm{supp}(f')$. For each $\alpha\in \mathrm{supp}(f'')$, define $f''(\alpha)\in V[H_\delta]$ such that $\Vdash_{E\restriction \alpha} f''(\alpha)\leq f'(\alpha)$ and $s'\restriction \alpha\Vdash_{E\restriction \alpha} f''(\alpha)\leq f'(\alpha)$. This can be achieved by applying the maximality principle in $V[H_\delta]$ to $E\restriction \alpha$.

In $V[H_\delta]$, let $\dot{R^*}$ be the $D$-name for the countably closed poset as above. In $V[H_\delta]$, let $\dot{T}$ be a $D$-name for a Baire tree. Then $D*(\dot{T}\times \dot{R^*})$ is $\omega$-distributive by Claim \ref{countablyClosed}. By Lemma \ref{indestructible}, in $V[H_\delta]$ we have $\Vdash_{D*(\dot{T}\times \dot{R^*})} E$ is c.c.c. This means in $V[G_\delta]$, $\Vdash_{T\times R^*} E$ is c.c.c. In particular, we know that in $V[G_\delta]$, $T\times R^*$ is Baire and $E$ is c.c.c. By Lemma \ref{capturing} and Remark \ref{remarkProper}, $\Vdash_E R^*\times T$ is $\omega$-distributive, hence $\Vdash_E \Vdash_{R^*} T$ is Baire. Since in $V[G_\delta]$, $E\times R^*$ projects onto $Q/G_\delta$, we know $\Vdash_{Q/G_\delta} T$ is Baire.

\end{proof}

\begin{proof}[Proof of Corollary \ref{separate}]
In fact, we show that in the forcing extension by $Q$, $s\mathrm{RC}^B$ holds and any $\aleph_1$-subtree of $\sigma(\mathbb{R})$ is special. The latter clearly implies the failure of $\mathrm{RC}$ by Observation \ref{observation}. Claim $\ref{projection}$ implies that $\omega_1$ is preserved.

Let $G$ be generic for $Q$. As $Q$ projects onto $P_\kappa$ (Claim \ref{firstCoordinate}), we can find $H\in V[G]$ that is $V$-generic for $P_\kappa$. Let $T\in V[G]$ be a Baire tree of size $\theta$. Without loss of generality, we may assume $T=(\theta, <_T)$. Let $\dot{T}$ be a $Q$-name for $T$. Let $j: V\to M$ witness that $\kappa$ is $\lambda$-supercompact for some sufficiently large regular cardinal $\lambda>\theta$. We may choose $\lambda$ large enough so that it is larger than the cardinality of any nice $Q$-name of a subset of $\theta$. Since $Q\subset V_\kappa$ is $\kappa$-c.c., we see that $j\restriction Q=id\restriction Q$ is a complete embedding. Hence we can view $Q$ as an initial segment of $j(Q)$. In fact, $Q=j(Q)\restriction \kappa$.

By the choice of $\lambda$, we know $\dot{T}\in M$. Hence $T\in M[G]$. By Claim \ref{quotient}, we know that in $M[G]$, $\Vdash_{j(Q)/G} T$ is Baire. Let $K\subset j(Q)/G$ be generic over $M[G]$, then we can lift $j$ to an elementary embedding $j^+: V[G]\to M[G*K]$. In $M[G*K]$, $|T|=\theta<j^+(\kappa)=(\omega_2)^{M[G*K]}$.  Since $j''\theta\in M[G*K]$ and $(T, <_T)$ is isomorphic to $(j'' \theta,<_{j(T)})$, we know that $M[G*K]\models$ there exists a subtree $A\subset j(T)$ such that $|A|\leq \aleph_1$ and $A$ is Baire. By the elementarity of $j^+$, the same statement is true in $V[G]$. We have then shown that $V[G]\models s\mathrm{RC}^B$.

Finally we show that any $\aleph_1$-subset of $\sigma(\mathbb{R})$ is special. Let $A\subset \sigma(\mathbb{R})$ be a $\aleph_1$-subset in $V[G]$. Note that $Q/H$ is $\omega$-distributive, so $(\sigma(\mathbb{R}))^{V[G]}=(\sigma(\mathbb{R}))^{V[H]}$. Since $Q$ is $\kappa$-c.c. and so is $P_\kappa$, $Q/H$ is $\kappa$-c.c in $V[H]$. Thus there exists $A'\in V[H]$ and $A'\subset \sigma(\mathbb{R})$ of size $<\kappa$ such that in $V[H]$, $\Vdash_{Q/H} \dot{A}\subset A'$, where $\dot{A}\in V[H]$ is a $Q/H$-name for $A$. As $A'$ is special in $V[H]$ (hence also in $V[G]$) and $A\subset A'$ in $V[G]$, $A$ is also special in $V[G]$.

\end{proof}

\section{Consistency of $s\mathrm{RC}^B + \mathrm{MA}_{\omega_1}(\mathrm{BIP})$}\label{iterate}

\begin{definition}
Fix a poset $P$, a sufficiently large cardinal $\lambda$, a countable $M\prec H(\lambda)$ containing $P$ and a countable sequence $\langle D_n: n\in \omega\rangle$ of dense subsets of $P\cap M$. 
We say $P$ is \emph{semi-strongly proper for $M$ and $\langle D_n: n\in \omega\rangle$} if for any $p\in P\cap M$, there exists $q\leq p$, such that for all $n\in \omega$, $q\Vdash D_n\cap \dot{G}\neq \emptyset$, where $\dot{G}$ is the canonical name of a generic filter. We say such $q$ is \emph{semi-strongly generic for $M$ and
$\langle D_n: n\in \omega\rangle$} (or just $\langle D_n: n\in \omega\rangle$ if $M$ is clear from the context).
\end{definition}

Note that in the definition above we do not require $D_n = D\cap M$ for some $D\in M$.

\begin{remark}
In the following, when the model $M$ is clear from the context, we will just say $P$ is semi-strongly proper for $\langle D_n: n\in \omega\rangle$.
\end{remark}

\begin{definition}\label{semi}
For any poset $P$, $P$ is \emph{semi-strongly proper} if for any sufficiently large $\lambda$, any $M\prec H(\lambda)$ countable containing $P$ and any countable sequence $\langle D_n: n\in \omega\rangle$ of dense subsets of $P\cap M$, $P$ is semi-strongly proper for $M$ and $\langle D_n: n\in \omega\rangle$.
\end{definition}

\begin{remark}
Note that the class of semi-strongly proper forcings here properly contains the class of strongly proper forcings in the sense of Mitchell \cite{MR2452816}. Strongly proper forcings always add Cohen reals (see the remark after Definition 1.8 in \cite{sidecon}) while any countably closed forcing will be semi-strongly proper. To see the latter, suppose $P$ is countably closed and $M, \langle D_n: n\in \omega\rangle, p$ are given as in Definition \ref{semi}. We can build a decreasing sequence $\langle p_n\in P\cap M: n\in \omega\rangle$ such that $p_0=p$ and $p_{n+1}\in D_n$ for all $n\in \omega$. The reason why we can do this is that each $D_n$ is a dense subset of $P\cap M$ for any $n\in \omega$. Any lower bound $p^*$ for $\langle p_n: n\in \omega\rangle$ will be the desired semi-strongly generic condition for $M$ and $\langle D_n: n\in \omega\rangle$. Shelah (\cite{MR1623206}, Chapter IX, Definition 2.6 in page 441) used the name ``strongly proper forcings'' to refer to what we call ``semi-strongly proper forcings'' here. We do this to avoid confusion.
\end{remark}

In general, Baire trees are preserved when forcing with semi-strongly proper posets.

\begin{lemma}\label{BaireAndStronglyProper}
Let $T$ be a Baire tree and $P$ be a semi-strongly proper poset. Then $\Vdash_T P$ is semi-strongly proper. In particular, after Remark \ref{remarkProper} and Lemma \ref{capturing}, $\Vdash_P T$ is Baire. 
\end{lemma}
\begin{proof}
Let $H\subset T$ be generic over $V$. In $V[H]$, let $\lambda$ be large enough and let countable $M\prec H(\lambda)$ be such that $M\cap H(\lambda)^V \prec H(\lambda)^V$ contains $P$. Let $p\in M\cap P$ and a sequence of dense subsets of $M\cap P$, say $\bar{D}=\langle D_n: n\in \omega\rangle$ be given. For any $n\in \omega$, since $D_n\subset M\cap P = (M \cap H(\lambda)^V)\cap P$ and is countable, by the countable distributivity of $T$, we know that $D_n \in V$. Applying the countable distributivity of $T$ one more time, we get that $\langle D_n: n\in \omega\rangle\in V$. Since in $V$, $P$ is semi-strongly proper for $M\cap H(\lambda)^V$, there exists $q\leq p$ that is semi-strongly generic for $M \cap H(\lambda)^V$ and $\bar{D}$, namely $q\Vdash_P D_n \cap \dot{G}\neq \emptyset$ for all $n\in \omega$. But this property persists to $V[H]$ by the absoluteness of the definability of forcing.
 
Notice what we have shown is that in $V[H]$, for sufficiently large $\lambda$, there exists a club subset of $[H(\lambda)]^\omega$ witnessing the semi-strong properness of $P$. By the standard trick (see for example Theorem 2.13 in \cite{MR2768684}), we can eliminate the club in the statement.
 
\end{proof}

We will use the theorem due to Feng in the following remark.

\begin{theorem}[Feng \cite{MR1760586}, Theorem 2.3]\label{feng}
Fix a poset $\mathbb{P}$.
Then $\mathbb{P}$ is $\omega$-distributive iff for any $p\in \mathbb{P}$, for any sufficiently large regular cardinal $\lambda$, the following set is stationary in $P_{\omega_1}(H(\lambda))$: 
$$S_p=\{M \prec H(\lambda): |M|=\aleph_0, \{\mathbb{P},p\}\in M, \exists p^*\leq p \ \forall \text{dense} \  D \subset \mathbb{P} $$ $$(D\in M)\rightarrow (\exists p'\in D\cap M \ p^*\leq_{\mathbb{P}} p')\}.$$
\end{theorem}

\begin{remark}
The reader may notice that we restrict our attention to proper forcings. This is natural since any forcing that preserves all Baire trees is necessarily proper. To see this, let $P$ be a given forcing that preserves all Baire trees. Let $\lambda\geq \omega_1$ be a given cardinal and $S\subset [\lambda]^\omega$ be a stationary subset. We may assume for any $x\in S$, $x\cap \omega_1\in \omega_1$. We need to show $P$ preserves the stationarity of $S$. Suppose for the sake of contradiction that there are a $P$-name  $\dot{f}: \lambda^{<\omega}\to \lambda$ and $p\in P$ such that $p\Vdash ``\forall x\in S$, $x$ is not closed under $\dot{f}$''.
For any $A, B\in [\lambda]^\omega$, define $A\lesssim B$ iff $A\subset B$, $A\cap \omega_1<B\cap \omega_1$ and $\sup (A) < \sup (B)$.

Consider the tree $T(S)$ defined as follows: $t\in T(S)$ iff there exists $\gamma<\omega_1$ such that $t: \gamma+1\to S$ is a continuous function satisfying that for any $\alpha<\beta\leq \gamma$, $t(\alpha)\lesssim t(\beta)$. The order in $T(S)$ is end extension.
For each $t\in T(S)$, let $\max(t)$ denote the $\lesssim$-maximal element in the image of $t$.

First note that in $V$, $T(S)$ is Baire (Todor\v{c}evi\'{c} \cite{MR657114}). Let $G\subset P$ containing $p$ be generic over $V$. By the hypothesis, $(T(S))^V$ is Baire in $V[G]$. We work in $V[G]$. Let $\theta$ be a sufficiently large regular cardinal. By Theorem \ref{feng}, $L=_{def}\{M\in [H(\theta)]^\omega: M\prec H(\theta), \exists t'\in (T(S))^V, \forall \text{dense }D\subset (T(S))^V (D\in M\rightarrow \exists t\in D\cap M \ t\leq_{(T(S))^V} t')\}$ is a stationary subset of $[H(\theta)]^\omega$. Fix $M\in L$ containing $\lambda, (T(S))^V$ such that $M\cap \lambda$ is closed under $f=(\dot{f})^G$. Let $\delta=_{def} M\cap \omega_1$.
For each $\alpha\in M\cap \lambda$ and $\beta\in \delta$, consider $D_{\alpha,\beta}=\{t\in (T(S))^V: \alpha\in \max (t), \mathrm{dom}(t)\geq \beta+1\}$ and notice that each $D_{\alpha,\beta}$ is a dense subset of $(T(S))^V$ in $M$. Let $t'\in (T(S))^V$ be such that for any $(\alpha,\beta)\in (M\cap \lambda) \times \delta$, there exists some $t_{\alpha,\beta}\in D_{\alpha,\beta}\cap M$ such that $t_{\alpha,\beta} \leq_{(T(S))^V} t'$. In particular, this implies that for any $\beta\in \delta$, $t'(\beta)\in M$ hence $t'(\beta)\subset M$ and for any $\alpha\in M\cap \lambda$, there is some $\beta\in \delta$ such that $\alpha\in t'(\beta)$. Consequently, by continuity, $t'(\delta)=\bigcup_{\beta<\delta} t'(\beta) = M\cap \lambda$. Since $t'\in (T(S))^V$, we know that $M\cap \lambda\in S$. Hence in $V[G]$, we have found an element in $S$ that is closed under $f$, which is a contradiction.
\end{remark}

\begin{remark}
It is a theorem of Shelah (Theorem 2.7 and Remark 2.7A in \cite{MR1623206}, Chapter IX, Page 441) that countable support iteration of semi-strongly proper forcings is semi-strongly proper. As a result, we can get the consistency of $s\mathrm{RC}^B + \mathrm{MA}_{\omega_1}(\text{semi-strongly proper})$ rather easily. To see this, start with a ground model with a supercompact cardinal $\kappa$. We can perform a countable support iteration of semi-strongly proper forcings guided by a Laver function of length $\kappa$. It follows as in the standard proof of the consistency of $\mathrm{PFA}$ (see Section 24 in \cite{MR2768691} for a proof) that the final model satisfies $\mathrm{MA}_{\omega_1}(\text{semi-strongly proper})$. The reason why the final model satisfies $s\mathrm{RC}^B$ is because any tail of the iteration is semi-strongly proper by the aforementioned theorem of Shelah, which in turn implies that it preserves Baire trees by Lemma \ref{BaireAndStronglyProper}. A standard reflection argument, which can be easily adapted from the proof of Theorem \ref{iterationSacks} presented at the end of this section, shows that $\mathrm{RC}^B$ also holds in this model.

However, it is not clear if $\mathrm{MA}_{\omega_1}(\text{semi-strongly proper})$ is strong enough to ensure the failure of $\mathrm{RC}$ because of the following restriction:  
\begin{claim}
The Baumgartner specializing forcing $P=_{def}S(\sigma(\mathbb{R}))$, as defined in Section \ref{separation} before Theorem \ref{Baumgartner}, is not semi-strongly proper.
\end{claim}

\begin{proof}
Suppose for the sake of contradiction that $P$ is semi-strongly proper and let $\dot{g}$ be the $P$-name for the generic specializing function. Let $\lambda$ be a sufficiently large regular cardinal such that $P, \dot{g}$ and all relevant parameters belong to $H(\lambda)$.

First we will find a countable $M\prec H(\lambda)$ containing $\dot{g}, P$ with $\delta=_{def} M\cap \omega_1$ and $t'\in \sigma(\mathbb{R})$ of height $\delta$ such that for any $\alpha<\delta$, $t'\restriction \alpha \in M$. Notice that for any $t\in \sigma(\mathbb{R})$, $r>\sup(t)$ and $\alpha>\mathrm{ht}(t)$, there exists an extension $t'$ of $t$ such that $r> \sup(t')$ and $\alpha=\mathrm{ht}(t')$. The reason is that for any non-empty open interval $(a,b)$ of the reals with the usual ordering embeds any countable ordinal. Fix some $r\in \mathbb{R}$. Construct recursively $\langle M_i \prec H(\lambda): i\in \omega\rangle$ and $\langle t_i: i\in \omega\rangle$ such that $\dot{g}, P\in M_0$ and
\begin{enumerate}
\item $M_i$ is countable,
\item $M_i\subset M_{i+1}$, $t_i\leq_{\sigma(\mathbb{R})} t_{i+1}$ and $t_i\in M_{i+1}$,
\item $\mathrm{ht}(t_i)=\delta_i$ where $\delta_i=_{def} M_i\cap \omega_1$,
\item $\sup (t_i)<r$.
\end{enumerate}
To see how the construction is carried out, suppose we have defined $M_i, t_i$ satisfying the above. Let $M_{i+1}\supset M_i$ be any countable elementary submodel of $H(\lambda)$ containing $t_i$. Let $t_{i+1}$ be an extension of $t_i$ such that $\mathrm{ht}(t_{i+1})=\delta_{i+1}$ and $\sup(t_{i+1})<r$. Let $M=\bigcup_{i\in \omega} M_i$ and $t'=\bigcup_{i\in\omega} t_i$. It is easy to check that $M$ and $t'$ are as desired.

Fix an increasing sequence of countable ordinals $\langle\delta_i: i\in \omega\rangle$ converging to $\delta$. For each $n\in \omega$, consider $D_n=\{p\in P\cap M: \exists i\in \omega\ p\Vdash \dot{g}(t'\restriction \delta_i)=n\}$. Each $D_n$ is a dense subset of $P\cap M$, since given $p\in P\cap M$, if it is not already true that there exists $i\in \omega$ such that  $p(t'\restriction \delta_i)=n$, then we can always find an extension of $p$ in $P\cap M$ that satisfies the property as the domain of $p$ is finite. By the assumption that $P$ is semi-strongly proper, we can find $q\in P$ such that $q\Vdash D_n\cap \dot{G}\neq \emptyset$ for all $n\in \omega$. Let $q'\leq q$ and $m\in \omega$ such that $q'\Vdash \dot{g}(t')=m$. Since $q'\Vdash D_m\cap \dot{G}\neq \emptyset$, there exists $q''\leq q'$ and $i\in \omega$ such that $q''\Vdash \dot{g}(t'\restriction \delta_i)=m$. This contradicts the fact that $\dot{g}$ is forced to be a specializing function.

\end{proof}
There are other natural examples of BIP forcings that are not semi-strongly proper, like the Laver forcing (\cite{MR2023448}, Corollary 4.1.7).
\end{remark}

\begin{remark}\label{restriction}
It may be tempting to conjecture that for a Baire tree $T$, any countable support iteration of proper forcings that preserve the Baireness of $T$ preserves the Baireness of $T$. However, in general this is false. In fact, it is consistent that there exists a Baire tree $T$ and a countable support iteration of proper forcings $\langle P_i, \dot{Q}_j: i\leq \omega, j<\omega\rangle$ such that $\Vdash_{P_i} T$ is Baire for all $i<\omega$, but $\Vdash_{P_\omega} T$ is special. For one such example, see the proofs of Lemma 4.6 and Theorem 4.7 in Shelah \cite{MR1623206}, Chapter IX, Pages 455 - 457.

\end{remark}

In light of Remark \ref{restriction}, we need to consider a stronger property that implies Baire-preserving so that this property is also preserved under countable support iteration. This class should also include semi-strongly proper forcings. The class $\mathrm{BIP}$ (see Definition \ref{BIP}) turns out to be as desired.

\begin{definition}\label{operations}
Fix $M\prec H(\lambda)$ countable containing all relevant objects, including a countable support iteration of proper forcings $\langle P_i, \dot{Q}_j: i\leq \alpha,j<\alpha\rangle$. Let $C$ be a countable collection of dense subsets of $P_i\cap M$ for possibly several $i\in M\cap \alpha+1$. We say $C$ is \emph{closed under operations} with respect to $M$ (if $M$ is clear from the context we will just say $C$ is \emph{closed under operations}) if for any $D\in C$, $\gamma<\gamma'\in M\cap \alpha+1$ such that $D$ is a dense subset of $P_{\gamma'}\cap M$ and any $(p,\dot{q})\in M\cap (P_{\gamma}*P_{[\gamma, \gamma')})$, the set $A_{D,\gamma, (p,\dot{q})}=\{r\in P_\gamma\cap M: r\perp p \vee \exists \dot{q}' \ (r'=_{def}(r,\dot{q}')\in D\cap M, r'\leq (p,\dot{q}))\}$ is also in $C$. We let $C_\gamma$ to denote the collection of $D\in C$ such that $D$ is a dense subset of $P_\gamma\cap M$. 
\end{definition}

\begin{remark}
In order for the definition above to make sense, we need to verify $A_{D,\gamma, (p,\dot{q})}$ as defined is dense in $P_\gamma\cap M$. But this is clear. 
\end{remark}

\begin{claim}\label{DenseCarried}
Let $M$ and $\langle P_i, \dot{Q}_j: i\leq \alpha,j<\alpha\rangle$ be as in Definition \ref{operations}. For any $C$ closed under operations, and $\gamma\in M\cap \alpha+1$, suppose $G\subset P_\gamma$ is generic over $V$ such that $G$ meets all the dense sets in $C_\gamma$, then in $V[G]$, for any $D\in C_{\gamma+1}$, the set $(D)^G=_{def}\{(\dot{q})^G: \exists p\in G \ (p,\dot{q})\in D\}$ is dense in $M[G]\cap Q_\gamma$.
\end{claim}

\begin{proof}
In $V[G]$, let $t\in M[G]\cap Q_\gamma$. Let $\dot{t}\in M$ be its name. Since $M[G]\prec H(\lambda)^{V[G]}$, we know that there exists $p\in G\cap M[G]\cap P_\gamma\subset M\cap P_\gamma$ such that $p\Vdash \dot{t}\in \dot{Q}_\gamma$. Consider $A_{D, \gamma, (p,\dot{t})}$, which is a set in $C_\gamma$ by the closure assumption. Since $G\cap A_{D, \gamma, (p,\dot{t})}\neq \emptyset$, we can pick $r$ in the intersection. By the definition of $A_{D, \gamma, (p,\dot{t})}$ and the fact that $r, p\in G$ which implies they are compatible, there exists $\dot{\tau}\in M$ such that $(r,\dot{\tau})\in D$ and is below $(p,\dot{t})$. But then since $r\in G$, $(\dot{\tau})^G \in (D)^G\subset M[G]\cap Q_\gamma$ and is stronger than $t$.
\end{proof}

Before we proceed with our iteration lemma, we need an extension lemma due to Shelah (\cite{MR1623206}, Theorem 2.7 and Remark 2.7A) about iteration of semi-strongly proper posets. Since the presentation in \cite{MR1623206} is in terms of $\aleph_1$-free limit and is relatively dense, we include a proof as a service to the reader.

\begin{lemma}[Shelah]\label{StronglyProperPreservation}
Let $\langle P_i, \dot{Q}_j: i\leq \alpha, j<\alpha\rangle$ be a countable support iteration of proper forcings and let $M\prec H(\lambda)$ be countable and contain all relevant objects including $P_\alpha$. Fix $\alpha_0\in M\cap \alpha$. Suppose $C$ is a countable collection of dense subsets of $P_\gamma\cap M$ for possibly several $\gamma\in M\cap (\alpha+1)$ closed under operations. 

Suppose for each $\gamma\in M\cap \alpha$ and $q\in P_\gamma$ that is semi-strongly generic for $M$ and $C_\gamma$, $q\Vdash_{P_\gamma} ``\dot{Q}_\gamma$ is semi-strongly proper for $M[\dot{G}_\gamma]$ and $(C_{\gamma+1})^{\dot{G}_\gamma}=_{def} \{(D)^{\dot{G}_\gamma}: D\in C_{\gamma+1}\}$''.

If $q\in P_{\alpha_0}$ is semi-strongly generic for $C_{\alpha_0}$ and $\dot{p}$ is a $P_{\alpha_0}$-name such that 

$$q\Vdash_{P_{\alpha_0}} \dot{p}\in P_{\alpha}\cap M, \dot{p}\restriction \alpha_0 \in \dot{G}_{\alpha_0},$$
then there exists $q'\in P_\alpha, q'\restriction \alpha_0 = q$ and $q'$ is semi-strongly generic for $C_{\alpha}$ and 

$$q'\Vdash_{P_\alpha} \dot{p}\in \dot{G}_\alpha.$$

\end{lemma}

\begin{proof}

We proceed by induction on $\alpha$. If $\alpha=\beta+1$, fix such $M\prec H(\lambda)$ countable containing the iteration and $\alpha_0\in M\cap \alpha$. Note that $\beta\in M$. Let $q\in P_{\alpha_0}$, and $\dot{p}\in V^{P_{\alpha_0}}$ be as given.
Apply the induction hypothesis, we get $q'\leq q, q'\in P_\beta, q'\restriction \alpha_0 =q$ such that $q'$ is semi-strongly generic for $C_{\beta}$ and $q'\Vdash_{P_\beta} ``\dot{p}\restriction \beta\in \dot{G}_\beta$ and $\dot{p}\in P_{\alpha}\cap M$''. By the hypothesis and Claim \ref{DenseCarried}, we have $q'\Vdash_{P_\beta} ``\dot{Q}_\beta$ is semi-strongly proper for $M[\dot{G}_\beta]$ and $(C_{\beta+1})^{\dot{G}_\beta}$''.
Let $G_\beta\subset P_\beta$ be generic over $V$ containing $q'$. Then in $V[G_\beta]$, $(\dot{p})^{G_\beta}=p\in P_\alpha\cap M$. Since in $V[G_\beta]$, $Q_\beta$ is semi-strongly proper for $M[G_\beta]$ and $(C_{\beta+1})^{G_\beta}$, there exists $t\leq_{Q_\beta} (p(\beta))^{G_\beta}$ that is semi-strongly generic for $(C_{\beta+1})^{G_\beta}$. Let $\dot{t}$ be a $P_\beta$-name for $t$ such that $q'$ forces it satisfies all the properties above, which exists by the maximality principle of forcing. Hence $(q',\dot{t})$ is the desired extension. Indeed, $(q',\dot{t})\Vdash_{P_\alpha} ``\dot{p}\in \dot{G}_\alpha$'', $(q',\dot{t})\restriction \alpha_0 = q'\restriction \alpha_0 = q$ and $(q',\dot{t})$ is semi-strongly generic for $C_{\alpha}$ (which is easily implied by the fact that $q'$ is semi-strongly generic for $C_\beta$ and $q'\Vdash ``\dot{t}$ is semi-strongly generic for $(C_{\beta+1})^{\dot{G}_\beta}$'').

When $\alpha$ is a limit, list $\{ D_n: n\in \omega\}$ in $C_\alpha$ and fix $\alpha_0\in M\cap \alpha, q\in P_{\alpha_0}$ and $\dot{p}$ as in the hypothesis. 
Fix an increasing $\langle \alpha_i\in M\cap \alpha: i\in \omega\rangle$ cofinal in $\sup (M\cap \alpha)$.

We build the following sequences: 
$\langle q_i: i<\omega\rangle$, $\langle \dot{p}_i: i<\omega\rangle$ such that 

\begin{itemize}
\item $\dot{p}_0 = \dot{p}, q_0 = q$,
\item $q_i\in P_{\alpha_i}$ is semi-strongly generic for $C_{\alpha_i}$,
\item $\dot{p}_i$ is a $P_{\alpha_i}$-name,
\item $q_{i+1}\Vdash_{P_{\alpha_{i+1}}} ``\dot{p}_{i+1}\in P_\alpha\cap M, \dot{p}_{i+1}\restriction \alpha_{i+1} \in \dot{G}_{\alpha_{i+1}}, \dot{p}_{i+1}\leq \dot{p}_{i}, \dot{p}_{i+1} \in D_{i}$'',
\item $q_{i+1}\restriction \alpha_i = q_i$.
\end{itemize}

If the construction is successful, then by the standard argument as in the properness preservation theorem (see for example \cite{MR2768684}, Lemma 2.8), we will be done.

Now we demonstrate for given $q_i, \dot{p}_i$, how to find $q_{i+1}, \dot{p}_{i+1}$ satisfying the requirements. Let $G_{\alpha_i}\subset P_{\alpha_i}$ be generic over $V$ containing $q_i$. In $V[G_{\alpha_i}]$, we have $p_i = (\dot{p}_i)^{G_{\alpha_i}} \in P_\alpha\cap M$ and $p_i\restriction \alpha_i \in G_{\alpha_i}$. Consider the set $A_{D_i, \alpha_i, p_i}\in C_{\alpha_i}$. By the semi-strong genericity of $q_i$ with respect to $C_{\alpha_i}$, $G_{\alpha_i}\cap A_{D_i, \alpha_i, p_i}\neq \emptyset$. Fix some $r\in G_{\alpha_i}\cap A_{D_i, \alpha_i, p_i}$. Then there exists $(r,\dot{q})\in D_i$ with $(r,\dot{q})\leq p_i$. 

By the maximality principle of forcing we can find a $P_{\alpha_i}$-name $\dot{p}_{i+1}$ such that $q_i\Vdash_{P_{\alpha_i}} ``\dot{p}_{i+1}\leq \dot{p}_i, \dot{p}_{i+1}\in D_i\subset M\cap P_\alpha$ and $ \dot{p}_{i+1}\restriction \alpha_i\in \dot{G}_{\alpha_i}$''.
Apply the induction hypothesis, we can get $q_{i+1} \in P_{\alpha_{i+1}}$ with $q_{i+1}\restriction \alpha_i =q_i$, $q_{i+1}$ is semi-strongly generic for $C_{\alpha_{i+1}}$ and $q_{i+1}\Vdash_{P_{\alpha_{i+1}}} ``\dot{p}_{i+1}\restriction \alpha_{i+1}\in \dot{G}_{\alpha_{i+1}}$''.

\end{proof}

\begin{cor}\label{DirectApplication}
Let $\langle P_i, \dot{Q}_j: i\leq \alpha, j<\alpha\rangle$ be a countable support iteration of proper forcing. Let $M\prec H(\lambda)$ be countable and contain $P_\alpha$. Let $C$ be a countable collection of dense subsets of $P_\gamma\cap M$ for possibly several $\gamma\in M\cap (\alpha+1)$ closed under operations. 

Suppose for each $\gamma\in M\cap \alpha$ and $q\in P_\gamma$ that is semi-strongly generic for $M$ and $C_\gamma$, $q\Vdash_{P_\gamma} \dot{Q}_\gamma$ is semi-strongly proper for $M[\dot{G}_\gamma]$ and $(C_{\gamma+1})^{\dot{G}_\gamma}$.

Then for each $p\in M\cap P_\alpha$, there exists $q\leq p$ that is semi-strongly generic for $C_\alpha$.
\end{cor}

\begin{proof}
Apply Lemma \ref{StronglyProperPreservation} with $\alpha_0=0$.
\end{proof}

Our main idea is that in order to prove properness of a poset in a countably distributive extension, it suffices to prove semi-strong properness in the ground model with respect to the relevant collection of dense sets.

The following iteration lemma is key to the proof of the main theorem.

\begin{lemma}[Key Lemma]\label{KeyLemma}
Let $T$ be a Baire tree and $\langle P_i, \dot{Q}_j: i\leq \alpha, j<\alpha\rangle$ be a countable support iteration of proper forcings such that for each $i<\alpha$, $\Vdash_{T\times P_i} \dot{Q}_i$ is proper. Then $\Vdash_{T} P_\alpha$ is proper.
\end{lemma}

\begin{remark}
Notice in Lemma \ref{KeyLemma} there is some abuse of notation, in that $\dot{Q}_i$ is actually a $P_i$-name, but it can be canonically identified as a $(T\times P_i)$-name, say $\dot{Q}'_i$. So here we really mean $\Vdash_{T\times P_i} \dot{Q}'_i$ is proper. But in general, there is only one way of interpretation based on the context, so we confuse $\dot{Q}_i$ with $\dot{Q}'_i$. We extend this practice to other similar situations.
\end{remark}

\begin{proof}
We proceed by induction on $\alpha$. If $\alpha=\beta+1$, then by the hypothesis, $\Vdash_{T\times P_\beta} \dot{Q}_\beta$ is proper. Let $H\subset T$ be generic over $V$. We need to show that $V[H] \models P_\alpha$ is proper. We will be done once we realize that in $V[H]$, $P_{\alpha}$ is a dense subset of $P_\beta * \dot{Q}_\beta$, since by the hypothesis, $V[H] \models P_\beta * \dot{Q}_\beta$ is proper.
\textbf{The difference between these two sets is }$P_\alpha$ is the two-step iteration defined in $V$, so $(p,\dot{q})\in P_\alpha\rightarrow (p,\dot{q})\in V$ while $P_\beta * \dot{Q}_\beta$ is the iteration defined in $V[H]$ which may contain $(p,\dot{q})$ where $\dot{q}\not\in V$. Given $(p,\dot{q})\in P_\beta * \dot{Q}_\beta$, we have $p\Vdash \dot{q}\in \dot{Q}_\beta$. Since $\dot{Q}_\beta$ is a $P_\beta$-name living in $V$, we know $p\Vdash \exists \dot{t}\in V \ \dot{t}=\dot{q}$. Let $p'\leq p$ and $\dot{t}\in V$ be such that $p'\Vdash \dot{t}=\dot{q}$. Then $(p', \dot{t})\leq (p,\dot{q})$ and $(p', \dot{t})\in P_{\alpha}$.

When $\alpha$ is a limit ordinal, let $H\subset T$ be generic over $V$. In $V[H]$, let $\lambda$ be a large enough regular cardinal, and $M\prec H(\lambda)$ be countable containing relevant objects including $P_\alpha$ such that $M\cap H(\lambda)^V\prec H(\lambda)^V$.
Note since $V[H]$ is a countably distributive extension of $V$, we have that $M'=_{def}M \cap H(\lambda)^V\in V$. 

For each $\gamma\in M\cap (\alpha+1)$, enumerate the dense subsets of $P_\gamma$ in $M$ as $\bar{D}^\gamma =\{D^\gamma_n: n\in \omega\}$. Let $\bar{D}^{\gamma}\restriction M = \{D^\gamma_n\cap M: n\in \omega\}$. As each $D^\gamma_n\cap M\subset P_\gamma\cap M =P_\gamma \cap M'\subset V$, by countable distributivity, we have $\bar{D}^\gamma\restriction M\in V$.

\begin{claim}\label{three}
$\bigcup_{\gamma\in M\cap (\alpha+1)} \bar{D}^\gamma\restriction M \in V$ is closed under operations.
\end{claim}

\begin{proof}[Proof of the Claim]
For any $\gamma'\leq \alpha$, $\gamma<\gamma' \in M$, $p\in P_{\gamma'}\cap M$ and any dense $D\subset P_{\gamma'}$ with $D\in M$, we need to show that $A_{D\cap M, \gamma, p}\in \bar{D}^\gamma\restriction M$. But this is immediate from the fact that $A'=\{r\in P_{\gamma}: r\perp p\restriction \gamma \vee \exists \dot{q}' \ r'=_{def}(r,\dot{q}')\in D, r'\leq p\}$ is a dense subset of $P_{\gamma}$ living in $M$ and $A_{D\cap M, \gamma, p}=A'\cap M \in \bar{D}^\gamma\restriction M$ by elementarity.
\end{proof}

\begin{claim}\label{two}
In $V$, for each $\gamma\in M\cap \alpha$ and any $q\in P_{\gamma}$ that is semi-strongly generic for $\bar{D}^\gamma\restriction M$, $q\Vdash_{P_{\gamma}} \dot{Q}_{\gamma}$ is semi-strongly proper for $M'[\dot{G}_\gamma]$ and $(\bar{D}^{\gamma+1}\restriction M)^{\dot{G}_\gamma}$.
\end{claim}
\begin{proof}[Proof of the claim]
Fix $\dot{r}\in M$ a $P_{\gamma}$-name for a condition in $\dot{Q}_\gamma$. 
 Let $G_\gamma\subset P_{\gamma}$ be generic over $V[H]$ containing $q$, then $V[H\times G_\gamma]\models Q_\gamma=(\dot{Q}_\gamma)^{G_\gamma}$ is proper with respect to $M[G_\gamma]$. Let $r\leq (\dot{r})^{G_\gamma}$ be a $(M[G_\gamma],  Q_\gamma)$-generic condition. We claim that in $V[G_\gamma]$, $r$ is semi-strongly generic for $(\bar{D}^{j+1}\restriction M)^{G_\gamma}$. Then the claim follows immediately.

The fact that, in $V$, $q$ is semi-strongly generic for $\bar{D}^\gamma\restriction M$ implies that in $V[H]$, $q$ is $(M, P_{\gamma})$-generic. Therefore, for each $D^{\gamma+1}_n\in \bar{D}^{\gamma+1}$,  $(D^{\gamma+1}_n)^{G_\gamma}$ is a dense subset of $Q_\gamma=(\dot{Q}_\gamma)^{G_\gamma}$ living in $M[G_\gamma]$. Fix $n\in \omega$.
Since in $V[H][G_\gamma]$, $r$ is $(M[G_\gamma], Q_\gamma)$-generic, we know $r\Vdash_{Q_\gamma} ``(D^{\gamma+1}_n)^{G_\gamma}\cap M[G_\gamma]\cap \dot{W}\neq \emptyset$'', where $\dot{W}$ is the canonical name for the generic filter. To see that this implies $r\Vdash_{Q_\gamma} ``(D^{\gamma+1}_n\cap M)^{G_\gamma}\cap \dot{W}\neq \emptyset$'', let $W\subset Q_\gamma$ be generic over $V[H][G_\gamma]$ containing $r$. In $V[H][G_\gamma][W]$, there exists $\dot{t}\in M$ such that $(\dot{t})^{G_\gamma}\in (D^{\gamma+1}_n)^{G_\gamma}$ and $(\dot{t})^{G_\gamma} \in W$. Since both $(\dot{t})^{G_\gamma}$ and $(D^{\gamma+1}_n)^{G_\gamma}$ are in $M[G_\gamma]\prec (H(\lambda))^{V[H][G_\gamma]}$, $M[G_\gamma]\models \exists p\in G_\gamma \exists \dot{l} \ (p,\dot{l})\in D_n^{\gamma+1}$ and $(\dot{t})^{G_\gamma} = (\dot{l})^{G_\gamma}$. Find $(p,\dot{l})\in D^{\gamma+1}_n\cap M[G_\gamma]$ witnessing the statement above. Since $q$ is $(M, P_\gamma)$-generic and $G_\gamma$ contains $q$, we know $M[G_\gamma]\cap V[H] = M$. So $(p,\dot{l})\in M$. Hence in $V[H][G_\gamma][W]$, 
we have found $(\dot{l})^{G_\gamma}\in (D^{\gamma+1}_n \cap M)^{G_\gamma}\cap W$.

We claim that $V[G_\gamma]$ models the same statement, namely $$r\Vdash_{Q_\gamma} (D^{\gamma+1}_n\cap M)^{G_\gamma}\cap \dot{W}\neq \emptyset.$$ Suppose not for the sake of contradiction. In $V[G_\gamma]$ we can extend $r$ to $r'$ to force the negation of the statement. Let $W\subset Q_\gamma$ containing $r'$ be generic over $V[H][G_\gamma]$. Then $V[G_\gamma * W] \models (D^{\gamma+1}_n\cap M)^{G_\gamma}\cap W = \emptyset$ but $V[H][G_\gamma * W]\models (D^{\gamma+1}_n\cap M)^{G_\gamma}\cap W \neq \emptyset$. By the product lemma, $V[G_\gamma * W]$ is a submodel of $V[H][G_\gamma * W]$ and the statement is absolute between these two models. We thus get a contradiction.

\end{proof}

Work in $V$. By Claim \ref{three}, Claim \ref{two} and Corollary \ref{DirectApplication}, we can conclude that $P_\alpha$ is semi-strongly proper for $M'$ and $\bar{D}^\alpha\restriction M$. 
Using the same argument as in Lemma \ref{BaireAndStronglyProper}, we conclude that in $V[H]$, $P_{\alpha}$ is proper with respect to $M$.

\end{proof}

\begin{theorem}\label{BIPpre}
Countable support iteration of $\mathrm{BIP}$ forcings is $\mathrm{BIP}$.
\end{theorem}

\begin{proof}
Let $\langle P_i, \dot{Q}_j: i\leq \alpha, j<\alpha\rangle$ be the iteration and $T$ be a given Baire tree. We show this by induction. If $\alpha=\beta+1$, then by the induction hypothesis, $\Vdash_T P_\beta$ is proper. In particular by Lemma \ref{capturing} and Remark \ref{remarkProper}, $\Vdash_{P_\beta} T$ is Baire. Since $\Vdash_{P_\beta} \dot{Q}_\beta$ is Baire indestructible, we have $\Vdash_{P_\beta} \Vdash_T \dot{Q}_\beta$ is proper. Hence by the product lemma, $\Vdash_T ``P_\beta \text{ is proper and }  \Vdash_{P_\beta} \dot{Q}_\beta$ is proper''. Since $\Vdash_T `` (P_{\beta+1})^V$ is a dense subset of $P_\beta * \dot{Q}_\beta$'', we see that $\Vdash_T `` P_{\beta+1}$ is proper''. If $\alpha$ is a limit, we check the hypothesis in the Key Lemma \ref{KeyLemma} is satisfied. For each $\beta<\alpha$, by the induction hypothesis, $\Vdash_T P_\beta$ is proper. Arguing as above, we have that $\Vdash_{T\times P_\beta} \dot{Q}_\beta$ is proper. Hence the hypothesis of Key Lemma \ref{KeyLemma} is satisfied, so we can conclude that $\Vdash_T P_\alpha$ is proper.
\end{proof}

\begin{cor}\label{BIPBaire}
Countable support iteration of $\mathrm{BIP}$ forcings preserve Baire trees.
\end{cor}
\begin{proof}
Immediately from Lemma \ref{capturing}, Remark \ref{remarkProper} and Theorem \ref{BIPpre}.
\end{proof}

\begin{proof}[Proof of Theorem \ref{iterationSacks}]\label{BIPproof}
Let $\kappa$ be a supercompact cardinal. Let $\langle P_\alpha, \dot{Q}_\alpha: \alpha<\kappa\rangle$ be the countable support iteration of BIP forcings guided by a Laver function of length $\kappa$. In this model, $\mathrm{MA}_{\omega_1}(\mathrm{BIP})$ holds and $\kappa=\omega_2$ (see Section 24 in \cite{MR2768691} for more details). We claim that $s\mathrm{RC}^B$ holds in this model. Let $G$ be a generic for $P_\kappa$. Let $T\in V[G]$ be a Baire tree of height $\omega_1$. Let the size of $T$ be $\theta$. Furthermore, we may assume $T$ is of the form $(\theta, <_T)$. Let $\dot{T}$ be a $P_\kappa$-name for $T$. Let $\lambda>> \max\{|\dot{T}|,\kappa,\theta\}$ be a sufficiently large regular cardinal.

Fix an elementary embedding $j: V\to M$ witnessing the $\lambda$-supercompactness of $\kappa$, namely, ${}^\lambda M\subset M$, $\mathrm{crit}(j)=\kappa$ and $j(\kappa)>\lambda$. Let $H\subset j(P_\kappa)/G$ be generic over $V[G]$. Then we can lift $j$ to an elementary embedding from $V[G]$ to $M[G][H]$. We will slightly abuse the notation by still using $j$ to refer to the lifted embedding in $V[G][H]$. Notice by the closure assumption, we have $j'' \theta, T \in M[G][H]$.

 By Corollary \ref{BIPBaire}, $T$ remains Baire in $M[G][H]$. Since $M[G][H]\models  (T, <_T)\simeq (j'' \theta, <_{j(T)})$, we know that $M[G][H]\models j(T)$ has a Baire subtree of size $<j(\kappa)$. By the elementarity of $j$, in $V[G]$, $T$ has a Baire subtree of size $<\kappa=\omega_2$.
\end{proof}

The following lemma gives yet another model separating $\mathrm{RC}^B$ from $\mathrm{RC}$.

\begin{lemma}\label{BIPimpliesNotRC}
$\mathrm{MA}_{\omega_1}(\mathrm{BIP})$ implies all $\aleph_1$ subtrees of $\sigma(\mathbb{R})$ are special.
\end{lemma}
\begin{proof}
This just follows from the observation that for any $\aleph_1$ subset $T'$ of $\sigma(\mathbb{R})$, the Baumgartner specializing forcing for $T'$ is BIP (see Theorem \ref{Baumgartner}).
\end{proof}

\section{The strength and limitations of the Baire Rado's Conjecture}\label{strength}

\begin{definition}
For any regular cardinal $\lambda$ and stationary $S\subset \lambda$, we say $S$ \emph{reflects} if there exists $\beta\in \lambda\cap \mathrm{cof}(>\omega)$ such that $S\cap \beta$ is stationary in $\beta$. Given a family $\mathcal{S}$ of stationary subsets of $\lambda$, we say $\mathcal{S}$ reflects simultaneously if there exists $\beta\in \lambda\cap \mathrm{cof}(>\omega)$ such that for each $S\in \mathcal{S}$, $S\cap \beta$ is stationary. 
\end{definition}

Todor\v{c}evi\'{c} (\cite{MR686495}, Theorem 8) showed that $\mathrm{RC}$ implies any stationary subset of $\lambda\cap \mathrm{cof}(\omega)$ reflects for any regular $\lambda\geq \omega_2$. The proof there uses some equivalent characterizations of $\mathrm{RC}$. We include a short argument here (from $\mathrm{RC}^B$ actually) using directly the tree formulation of $\mathrm{RC}$ as in Definition \ref{RC} and \ref{baireRC}. It is worth noting that Sakai derives the same conclusion from the Semistationary Reflection Principle (see \cite{MR2387938}), which is a consequence of $\mathrm{RC}^B$ by Theorem \ref{Doebler}.

\begin{theorem}[Todor\v{c}evi\'{c}]\label{ORefl}
$\mathrm{RC}^B$ implies any stationary subset of $\lambda\cap \mathrm{cof}(\omega)$ reflects for regular cardinal $\lambda\geq \omega_2$.
\end{theorem}
\begin{proof}
Let $S\subset \lambda\cap \mathrm{cof}(\omega)$ be given. Let $T(S)$ be the tree consisting of functions $t: \gamma+1\to S$ for some $\gamma<\omega_1$, such that $t$ is increasing and continuous, ordered by end extension. For $t\in T(S)$, let $\max (t)$ denote the maximal element in the image of $t$. The stationarity of $S$ implies $T(S)$ is Baire by the standard argument (see Theorem 23.8 in \cite{MR1940513}). For any subtree $T\subset  T(S)$, let $\sup T$ be $\sup \{\max (t): t\in T\}$. 

Apply $\mathrm{RC}^B$ to $T(S)$ and pick some $\aleph_1$-sized nonspecial subtree $T\subseteq T(S)$ with the least supremum, say $\delta$. The ordinal $\delta$ must have cofinality $\omega_1$. To see this, suppose otherwise there exists an increasing sequence $\langle \delta_i: i\in \omega\rangle$ converging to $\delta$. By the minimality of $\delta$, for each $i\in \omega$, $T_i=_{def} \{t\in T: \max(t)\leq \delta_i: i\in \omega\}$ is special. Notice that $T'=\{t\in T: \max (t)=\delta\}$ is also special since it is an antichain. Hence $T=\bigcup_{i\in \omega} T_i \cup T'$ is a countable union of special trees, which implies that $T$ is also special. This contradicts our assumption.

We claim that $S\cap \delta$ is stationary. Suppose not for the sake of contradiction, then there exists a club $C\subset \delta$ that is disjoint from $S\cap \delta$. Without loss of generality, we can assume $T$ is downward closed. Define a pressing down function $f: T\backslash \{\emptyset\}\to T$ where for each $t\in T$ of limit height, $f(t)=s$ if $s$ is the  $<_T$-least predecessor of $t$ such that $(\max (s), \max (t))\cap C =\emptyset$. By the Pressing Down Lemma for trees, there exists a nonspecial subtree $T'\subset T$ such that $f$ gets constant value $s$ with $\gamma=\max (s)$. Then for each $t\in T'$, $\max (t)\leq \min (C\backslash \gamma)<\delta$. Hence $\sup (T') <\delta$. By the minimality of $\delta$, $T'$ is special. This is a contradiction.
\end{proof}

\begin{remark}
The theorem above shows that $\mathrm{RC}^B$ implies the ordinary stationary reflection at ordinals of cofinality $\omega_1$. In general, $\mathrm{RC}$ does not imply stationary reflection at ordinals of cofinality $>\omega_1$. In fact $\mathrm{RC}$ is compatible with the fact that $\aleph_{\omega+1}\cap \mathrm{cof}(\geq \omega_2)$ carries a partial square (see the remark proceeding the Claim in the section 5 of \cite{MR1450520}, page 192), which in turn implies the failure of the ordinal stationary reflection at ordinals of cofinality $\geq \omega_2$ (see Theorem 4.2 in \cite{MR2768694}).
\end{remark}

Similar to the argument above, we are able to present an alternative argument that $\mathrm{RC}^B$ implies Semi-stationary Reflection (SSR) due to Doebler \cite{MR3065118}.

\begin{theorem}[Doebler, \cite{MR3065118}]\label{Doebler}
$\mathrm{RC}^B$ implies $\mathrm{SSR}$, where the latter means that for any regular cardinal $\lambda\geq \omega_2$ and any stationary $S\subset [\lambda]^\omega$ upward closed under $\sqsubset$ ($x\sqsubset y$ iff $x\subset y$ and $x\cap \omega_1 = y\cap \omega_1$), there exists $W\in [\lambda]^{\omega_1}$ containing $\omega_1$ such that $S\cap [W]^\omega$ is stationary.
\end{theorem}
\begin{proof}

Fix $\lambda, S$ as above. We may assume for any $x\in S$, $x\cap \omega_1\in \omega_1$.

Build the tree $T(S)$ consisting of countable continuous $\lesssim$-increasing sequences of elements in $S$, where for any $a\neq b\in S$, $a\lesssim b$ iff $a\subset b$, $a\cap \omega_1 < b\cap \omega_1$ and $\sup (a) < \sup (b)$. Hence by design, each element $t$ in $T(S)$ has a $\subset$-maximum element. Let $\max (t)$ denote this element. This tree is clearly Baire by the fact that $S$ is stationary.

Apply $\mathrm{RC}^B$, we can find a subtree $T'\subset T$ such that $T'$ is nonspecial. Let $W=\bigcup_{t\in T'} \max (t)$. We can choose such a $T'$ such that $\sup (W)$ is the least. Notice that $W\supset \omega_1$ and $\mathrm{cf}(\mathrm{sup}(W))>\omega$. This follows from the fact that $T'$ is nonspecial and the minimality of $\mathrm{cf}(\mathrm{sup}(W))$.

We claim that $S\cap [W]^\omega$ is stationary. Suppose not, then there exists a function $F: W^{<\omega}\to W$ such that $cl_F \cap S =\emptyset$. We may redefine $F$ such that for any $y\in [W]^\omega$, $F'' y^{<\omega} =cl_F(y)$. For each $t\in T'$, we have $\max (t) \in S$, so by our assumption, $\max (t)$ is not closed under $F$. In particular, $F'' \max (t) ^{<\omega} \not \in S$. Since $S$ is upward closed under $\sqsubset$, we know that $F'' \max (t) ^{<\omega} \cap \omega_1 > \max (t)\cap \omega_1$. Hence there exists $\bar{a}_t\in \max (t)^{<\omega}$ such that $F(\bar{a}_t)\geq \max (t)\cap \omega_1$ and $F(\bar{a}_t)\in \omega_1$. We can use this fact to define a regressive function, mapping each $t\in T'$ of limit height to the $<_T$-least $s\leq_T t$ such that $\bar{a}_t \subset \max (s)$. By the Pressing Down Lemma for nonspecial trees and the fact that special  trees are closed under countable union, there exists $T''\subset T'$ nonspecial and $\bar{a}\in W^{<\omega}$ with $F(\bar{a})\in \omega_1$ such that for each $t\in T''$, $F(\bar{a})\geq \max (t) \cap \omega_1 $. But this means that the height of $T''$ is bounded above by $F(\bar{a})$, which is a contradiction to the fact that $T''$ is nonspecial.

\end{proof}

From the aspect of simultaneous reflection, Theorem \ref{ORefl} is  optimal.

\begin{theorem}\label{NotSimul}
$\mathrm{RC}^B$ does not imply any two stationary subsets of $\omega_2\cap \mathrm{cof}(\omega)$ reflects simultaneously.
\end{theorem} 

\begin{proof}
We first prepare the ground model $V$ such that it satisfies that $\mathrm{RC}^B$ is indestructible under $\omega_2$-directed closed forcings. We can do this by Lévy collapsing $\kappa$ to $\omega_2$ where $\kappa$ is a supercompact cardinal. The heart of the argument is the analysis of the existence of a master condition needed to lift an elementary embedding in a countably closed generic extension. We refer the reader to Cummings' Chapter \cite{MR2768691} (specifically sections 9, 12, 14) in the Handbook of Set Theory for more details.

Let $\mathbb{P}$ be the standard poset that adds two stationary subsets of $\omega_2\cap \mathrm{cof}(\omega)$ that do not reflect simultaneously. More precisely, $p\in \mathbb{P}$ iff $p=(p_0, p_1)$ where
\begin{itemize}
\item $p_0$ and $p_1$ are partial functions from $\omega_2$ to $2$, taking value $1$ only on ordinals of countable cofinality, 
\item $\mathrm{dom}(p_0)=\mathrm{dom}(p_1)<\omega_2$, 
\item for no $\alpha<\omega_2$, $p_0(\alpha) = p_1 (\alpha)=1$,
\item for all $\beta \in \omega_2\cap \mathrm{cof}(\omega_1)$ and $\beta\leq \mathrm{dom}(p_0)$, there exists a club subset $C\subset \beta$ such that there is $i<2$, for any $\alpha\in C$, we have $p_i(\alpha)=0$.
\end{itemize} 
The elements in $\mathbb{P}$ are ordered by the coordinate-wise reverse inclusion. For $p\in \mathbb{P}$, we will sometimes abuse the notation by using $\mathrm{dom} (p)$ to mean $\mathrm{dom} (p_0)$ in the following.
The proof in Example 6.5 of \cite{MR2768691} easily adapts to show that $\mathbb{P}$ is $\omega_2$-strategically closed and $\mathbb{P}$ adds the characteristic functions of two disjoint stationary subsets of $\omega_2\cap \mathrm{cof}(\omega)$ that do not reflect simultaneously.

Let $\dot{S}_0, \dot{S}_1$ be the $\mathbb{P}$-name for the two stationary sets that are added by $\mathbb{P}$. Define in $V^\mathbb{P}$ forcings $\dot{Q}_i$ for shooting a club with closed initial segments through the complement of $\dot{S}_i$ for each $i<2$. It is a standard fact that $\mathbb{P}*\dot{Q}_i$ has a dense $\omega_2$-directed closed subset for any $i<2$. To see this, fix $i<2$. We may assume $i=0$ since the argument for the other case is symmetric. Consider $A=\{(p, B)\in \mathbb{P}* \dot{Q}_0 : \mathrm{dom} (p_0)=\max B+1, p_0\restriction B\equiv 0\}$. First we check that $A$ is $\omega_2$-directed closed. Suppose $D=\{(p^\alpha, B_\alpha)\in A : \alpha< \omega_1\}$ is a given directed family. Let $p^\infty=(p_0^\infty, p^\infty_1)$ where $p^\infty_i =\bigcup_{\alpha<\omega_1} p^\alpha_i$ for $i<2$ and $B_{\infty} = \bigcup_{\alpha<\omega_1} B_{\alpha}$. Let $\delta=\mathrm{dom}(p^\infty)$. If there exists $\alpha<\omega_1$ such that $\mathrm{dom} (p^\alpha) =\delta$, then we are done since $(p^\alpha, B_\alpha)$ will be the strongest condition in $D$. Otherwise, $B_\infty$ is a closed unbounded subset of $\delta$ such that $p^\infty_0 \restriction B_\infty \equiv 0$. Then $((p^\infty_0 \cup \{(\delta, 0)\}, p^\infty_1 \cup \{(\delta, 0)\}), B_\infty \cup \{\delta\})$ will be a desired lower bound for $D$. Next we check that $A$ is a dense subset of $\mathbb{P}*\dot{Q}_i$. Given $(p,\dot{q})\in \mathbb{P}*\dot{Q}_i$, since $\mathbb{P}$ is $\omega_2$-strategically closed, there exist $p^0\leq_{\mathbb{P}} p$ and $B_0\in V$ such that $p^0\Vdash_{\mathbb{P}} B_0=\dot{q}$. We may assume that $\mathrm{dom}(p^0)>\max (B_0)$. Let $\gamma>\mathrm{dom} (p^0)$. We can extend $(p^0, B_0)$ to $(p^1, B_1)$ by putting $\gamma$ into $B_1$ and setting $p^1_i(\alpha)$ to $0$ for any $i<2$ and any $\mathrm{dom}(p^0)\leq \alpha \leq \gamma$. Then $(p^1, B_1)\leq_{\mathbb{P}*\dot{Q}_0} (p,\dot{q})$ and $(p^1, B_1)\in A$.

Let $G\subset \mathbb{P}$ be generic over $V$. We show that $V[G]$ is the desired model. To this end, fix a Baire tree $T\in V[G]$. 
\begin{claim}\label{remain}
In $V[G]$, it is forced by $T$ that there exists $i<2$, such that $\omega_2^V\cap \mathrm{cof}^V(\omega)-S_i$ remains stationary.
\end{claim}

\begin{proof}
Otherwise, we can find $H\subset T$ generic over $V[G]$ such that in $V[G*H]$, for each $i<2$, there is a club $C_i\subset S_i\cup (\omega_2^V\cap \mathrm{cof}^V(\omega_1))$. As $T$ is Baire, $cf^{V[G*H]}(\omega_2^V)>\omega$. Hence in $V[G*H]$, the fact that $S_0\cap S_1=\emptyset$ implies that $C_0\cap C_1\subset \omega_2^V\cap \mathrm{cof}^V(\omega_1)$ is a club. But this implies that some ordinal whose cofinality is $\omega_1$ in $V$ now has cofinality $\omega$ in $V[G*H]$. This contradicts the $\omega_1$-distributivity of $\mathbb{P}$ in $V$ and the Baireness of $T$ in $V[G]$.
\end{proof}

\begin{claim}\label{preserve}
In $V[G]$, it is forced by $T$ that there exists $i<2$, $\dot{Q}_i$ is $\omega$-distributive.
\end{claim}
\begin{proof}
Work in $V[G]$. For any given $t'\in T$, use Claim \ref{remain} to find $t\leq t'$ and $i<2$ such that $t\Vdash_T \omega_2^V\cap \mathrm{cof}^V(\omega)-S_i$ is stationary. Let $H\subset T$ be generic over $V[G]$ containing $t$. We show $V[G*H]\models Q_i=(\dot{Q}_i)^{V[G]}$ is $\omega$-distributive. Let $\dot{\tau}$ be a $Q_i$-name for a $\omega$-sequence of ordinals. Let $M\prec H(\lambda)$ be countable and contain relevant objects such that $\gamma=\sup (M\cap \omega_2^V) \not \in S_i$ where $\lambda$ is some large enough regular cardinal. Fix a generic sequence $\langle p_i: i\in \omega\rangle$ for $M$ so in particular for any $j<\omega$ there exists some $i<\omega$ such that $p_i$ decides $\dot{\tau}(j)$ and $\sup_{i\in \omega} (p_i) = \gamma$. As $T$ is Baire in $V[G]$, $\langle p_i: i\in \omega\rangle \in V[G]$. Since $\gamma\not \in S_i$, we see that $\bigcup_{i\in \omega} p_i \cup \{\gamma\}\in Q_i$, which decides $\dot{\tau}$.
\end{proof}

By Claim \ref{preserve}, in $V[G]$, find $t\in T$ and $i<2$ such that $t\Vdash_T Q_i$ is $\omega$-distributive.  Define $T\restriction t $ to be $\{s\in T: s\leq t \vee t\leq s\}$.
By the Product Lemma, $\Vdash_{Q_i} T\restriction t$ is Baire. Let $R$ be generic for $Q_i$ over $V[G]$. Since $\mathbb{P}*\dot{Q}_i$ has a $\omega_2$-directed closed dense subset, it follows that $V[G*R]\models \mathrm{RC}^B$. Hence there exists a nonspecial $T'\leq T\restriction t$ of size $\aleph_1$. Since over $V[G]$, $Q_i$ is $\omega_1$-distributive, we know that $T'\in V[G]$ and it remains nonspecial in $V[G]$.

\end{proof}

\begin{definition}\label{ADS}
Let $\mu$ be a cardinal. $\langle A_i\subset \mu : i<\mu^+\rangle$ is said to be an almost disjoint sequence if for each $i<\mu^+$, $A_i$ is unbounded in $\mu$ and for each $\beta<\mu^+$, there exists $F: \beta \to \mu$ such that $\langle A_{i}\backslash F(i): i<\beta\rangle$ is pairwise disjoint. $\mathrm{ADS}_\mu$ abbreviates the assertion that there exists such a sequence.
\end{definition}

The interesting case of the above principle is when $\mu$ is singular.
It is known (\cite{MR1838355}, Theorem 4.1) $\mathrm{ADS}_\mu$ follows from the existence of a PCF-theoretic object called a \emph{better scale} at $\mu$, which is in turn a consequence of $\square^*_\mu$. It is a theorem of Shelah \cite{MR2369124} that if $\mathrm{SCH}$ fails, then the least ordinal where it fails (whose cofinality is necessarily $\omega$ by a theorem of Silver) carries a better scale. On the other hand, Sakai and Velickovic in \cite{MR3284479} show that the Semistataionary Reflection principle implies there is no better scale, extending the theorem of Todor\v{c}evi\'{c} \cite{MR1261218} that $\mathrm{RC}$ implies $S\mathrm{CH}$. 

The fact that for singular $\mu$ of countable cofinality, $\neg \mathrm{ADS}_\mu$ follows from $\mathrm{RC}$ is known: Torres-Pérez and Todor\v{c}evi\'{c} (\cite{MR2965421}, the proof of Theorem 3.1) showed this via an equivalent form of $\mathrm{RC}$ characterized in \cite{MR686495}; Fuchino, Juh\'{a}sz, Soukup, Szentmikl\'{o}ssy, and Usuba \cite{MR2610450} showed this via an intermediate principle called the \emph{Fodor-type Reflection Principle} or FRP (Fuchino \cite{Fuchino} showed this principle is indeed a consequence of $\mathrm{RC}^B$).
We present an alternative proof of $\neg \mathrm{ADS}_\mu$ directly from the tree formulation of $\mathrm{RC}^B$ using the same ideas as those in Theorem \ref{ORefl} and \ref{Doebler}.

\begin{prop}
$\mathrm{RC}^B$ implies $\neg \mathrm{ADS}_\mu$ for all singular cardinal $\mu$ with $cf(\mu)=\omega$.
\end{prop}

\begin{proof}
Suppose for the sake of contradiction that there exists $\langle A_i: i<\mu^+\rangle$ that witnesses $\mathrm{ADS}_\mu$. We may assume that each set in the sequence has order type $\omega$. Notice that $S=\{x\in [\mu^+]^\omega: A_{\sup (x)} \subset x\}$ is stationary. See the proof of Theorem 4.2 in \cite{MR1838355}, page 52. We may without loss of generality assume that for any $x\in S$, $x\cap \omega_1\in \omega_1$.
Define $T(S)$ as the tree for shooting a continuous $\omega_1$-sequence through $S$. More precisely, $t\in T(S)$ iff there exists $\gamma<\omega_1$ and $t: \gamma+1\to S$ that is continuous increasing such that for any $\alpha<\beta\leq \gamma$, $\sup (t(\alpha))<\sup (t(\beta))$ and $t(\alpha)\cap \omega_1<t(\beta)\cap \omega_1$. The order in $T(S)$ is end extension. For each $t\in T(S)$ with domain $\gamma+1$, let $\max (t) $ be $t(\gamma)$.
\begin{claim}
$T(S)$ is Baire.
\end{claim}
\begin{proof}
Fix a countable collection of dense open subsets of $T(S)$, say $\{D_i: i\in\omega\}$.
Let $\lambda$ be a large enough regular cardinal and $M\prec H(\lambda)$ be countable and contain relevant objects such that $M\cap \mu^+\in S$. Now build a sequence $\langle t_i\in T(S)\cap M: i\in \omega\rangle$ such that
\begin{itemize}
\item  $t_{2i}\in D_i$ for all $i\in \omega$,
\item $t_i \leq_{T(S)} t_{i+1}$ for all $i\in \omega$,
\item $\langle \mathrm{dom}(t_i): i\in \omega\rangle$ is cofinal in $M\cap \omega_1$,
\item $\bigcup_{i\in \omega} \max (t_i)=M\cap \mu^+$.
\end{itemize}  
It is easy to see that $\bigcup_{i\in \omega} t_i \cup \{\langle M\cap \omega_1, M\cap \mu^+\rangle\}$ is in $\bigcap_{i\in \omega} D_i$.
\end{proof}

Apply $\mathrm{RC}^B$ to $T(S)$ and pick a nonspecial subtree $T'\subset  T(S)$ of size $\aleph_1$.
Let $W=\{\sup(\max (t)): t\in T'\}$. By the almost-disjointness, there exists $F: W\to \mu$ such that $\langle A_i\backslash F(i): i\in W\rangle$ is pairwise disjoint. 

Define a function $f: T'\to T'$ such that for each $t\in T'$,
\begin{enumerate}
\item if $t$ is of successor height, then $f(t)$ is its immediate predecessor in $T'$,
\item if $t$ is of limit height $f(t)$ is the $<_{T'}$-least $s\leq_{T'} t$ such that there exists $\alpha\in \max(s)$ such that $\alpha\in A_{\sup (\max (t))}\backslash F(\sup (\max (t)))$. 
\end{enumerate}

Note that $f$ is a regressive function: for each $t\in T'$ of limit height, by continuity, $\max (t)=\bigcup_{s<_{T'} t} \max(s)$. Since $\max (t)\in S$, we know that $A_{\sup(\max(t))}\subset \max(t)$. Therefore, if $\alpha$ is the least element in $A_{\sup (\max (t))}\backslash F(\sup (\max (t)))\subset \max(t)$, there will be some $s<_{T'} t$ such that $\alpha\in \max(s)$.

By the Pressing Down Lemma for trees and the countable completeness of nonspecial trees, we can find a nonspecial subtree $T''\subset T'$ and $\alpha\in \mu^+$ such that for each $t\in T''$, $\alpha\in A_{\sup (\max (t))}\backslash F(\sup (\max (t)))$. By the fact that $\langle A_i\backslash F(i): i\in W\rangle$ is pairwise disjoint, we know that for any $t,t'\in T''$, $\sup(\max(t))=\sup(\max(t'))$. This means that $T''$ is an antichain, hence special, which is a contradiction.
\end{proof}

\begin{remark}
Cummings, Foreman and Magidor (\cite{MR1838355}, Theorem 4.2) showed that $\mathrm{WRP}^*([\mu^+]^\omega)$ implies the failure of $\mathrm{ADS}_{\mu}$, where $\mathrm{WRP}^*([\mu^+]^\omega)$ is the same as $\mathrm{WRP}([\mu^+]^\omega)$ (see Definition \ref{WRP}) except that in addition we require that the reflected set $W\in [\lambda]^{\omega_1}$ to satisfy that $\mathrm{cf}(\sup(W))=\omega_1$. Recall the discussion after Theorem \ref{MitchellRC} that in general $\mathrm{RC}$ does not imply $\mathrm{WRP}$.
\end{remark}

\begin{remark}
It follows from the work of Foreman and Magidor (\cite{MR1450520}, Section 5, the Claim in page 16 and the discussion preceding it) that $\mathrm{RC}$ is compatible with the Approachability Property at $\mu$ where $\mu$ is a singular cardinal of countable cofinality.
\end{remark}

The last part of this section is dedicated to discussing the interaction between $\mathrm{RC}$ and certain weak square principles.

\begin{definition}
For any set of ordinals $C$, define $\lim (C)$ to be the collection of accumulation points in $C$. Namely, $\lim (C)$ contains all $\gamma$ such that $\gamma=\sup (C\cap \gamma)$.
\end{definition}

\begin{definition}
Let $\lambda$ be a regular cardinal and $\kappa$ be a cardinal. $\square(\lambda,\kappa)$ asserts the existence of a sequence $\langle \mathcal{C}_\alpha: \alpha\in \lim (\lambda)\rangle$ such that 
\begin{itemize}
\item $\mathcal{C}_\alpha$ is a non-empty $\leq \kappa$-collection of clubs in $\alpha$ for each $\alpha\in \lim (\lambda)$,
\item for each $\alpha<\beta \in \lim (\lambda)$ and $C\in \mathcal{C}_\beta$, if $\alpha\in \lim (C)$, then $C\cap \alpha\in \mathcal{C}_\alpha$,
\item there does not exist a \emph{thread}, namely a club $D\subset \lambda$ such that for any $\alpha\in \lim (D)$, $D\cap\alpha\in \mathcal{C}_\alpha$.
\end{itemize}
\end{definition}

We will show that in general $\mathrm{RC}$ is compatible with $\square(\lambda,\omega_2)$ for any given regular $\lambda\geq \omega_3$. For uncountable regular cardinals $\lambda, \kappa$ such that $\lambda \geq \kappa^+$, we will define a $\kappa$-directed closed forcing which adds a $\square(\lambda,\kappa)$-sequence.

\begin{definition}
Let $\mathbb{P}=\mathbb{P}_{\square(\lambda,\kappa)}$ be the poset consisting of functions $t$ where
 $t$ is a function of domain $(\gamma+1)\times \kappa$ for some $\gamma=\gamma_t<\lambda$ such that for all $\beta\leq \gamma$ and $i<\kappa$, $t(\beta, i)$ is a club in $\beta$ and for any $\alpha<\beta \in \lim (\lambda) \cap \gamma+1$, $i<\kappa$ and $a = t(\beta,i)$, if $\alpha\in \lim (a)$, then $a\cap \alpha\in \bigcup_{j<\kappa} \{t(\alpha,j)\}$. We define $t'\leq_\mathbb{P} t$ iff $t'$ end extends $t$ and there exists $\eta<\kappa$ such that for all $\nu>\eta$, $t(\gamma_t, \nu)= t'(\gamma_{t'}, \nu)\cap \gamma_{t}$.

\end{definition}

Let us collect some standard facts about this forcing.

\begin{fact} 
The following hold:
\begin{enumerate}
\item $\mathbb{P}$ is $\kappa$-directed closed,
\item $\mathbb{P}$ is $\lambda$-strategically closed,
\item forcing with $\mathbb{P}$ adds a $\square(\lambda,\kappa)$-sequence.
\end{enumerate}
\end{fact}

\begin{proof}

\begin{enumerate}

\item To see the poset is $\kappa$-directed closed, fix any directed collection $\langle t_i\in \mathbb{P}: i<\beta\rangle$ with $\beta<\kappa$. Since the ordering is end-extension, the collection must be linearly ordered. Let $t'=\bigcup_{i<\beta} t_i$. If this is a condition, then we are done. Otherwise, let $\mathrm{dom}(t')=\gamma=\sup_{i<\beta}(\gamma_{t_i})$. We need to extend $t'$ to $t$ whose domain is $(\gamma+1)\times \kappa$. Clearly, we only need to specify the values of $t$ on $\{\gamma\}\times \kappa$. For each $i<\beta$, there exists $j_i\in \kappa$ such that for all $k>j_i$, for all $l<i$, $t_l(\gamma_{t_l}, k)=t_i(\gamma_{t_i}, k)\cap \gamma_{t_l}$. Let $\nu =\sup_{i<\beta} (j_i)$. For each $\nu'>\nu$, let $t(\gamma, \nu')=\bigcup_{i<\beta} t_{i}(\gamma_{t_i}, \nu')$ and for each $\mu\leq \nu$, let $t(\gamma, \mu)=t(\gamma,\nu+1)$ . It is easy to see that $t$ as defined is a desired lower bound for $\langle t_i\in \mathbb{P}: i<\beta<\kappa\rangle$.

\item Since the proof is standard, we direct the the reader to the proof of Lemma 6.1 in \cite{MR1838355} adapted suitably in the context of $\mathbb{P}$. 

\item We only need to verify that that there does not exist a thread for the generic sequence in the forcing extension. Suppose for the sake of contradiction that for some $\mathbb{P}$-name $\dot{C}$ for a club subset of $\lambda$ and $t\in \mathbb{P}$, $t\Vdash ``\dot{C}$ is a thread ''. 
Since $\mathbb{P}$ is $<\lambda$-distributive, we can recursively build $\langle t_i\in \mathbb{P}: i\in \omega\rangle$ and $\langle a_i\in [\lambda]^{<\lambda}: i\in \omega\rangle$ such that 
\begin{itemize}
\item $t_0=t$.
\item for any $i\in \omega$, $t_{i+1}\leq_{\mathbb{P}} t_i$, $a_i$ is a proper initial segment of $a_{i+1}$ and $a_i$ is a closed subset of $\lambda$,
\item for any $i\geq 1$, $\gamma_{t_{i}}>\max (a_{i})>\gamma_{t_{i-1}}$,
\item for any $i\in \omega$, $t_{i+1}\Vdash ``a_i\sqsubset \dot{C}$ and $\max (a_i)\in \lim (\dot{C})$'' and for any $\alpha<\alpha'\in \kappa$, $t_{i+1}(\gamma_{t_{i+1}}, \alpha)\neq t_{i+1}(\gamma_{t_{i+1}}, \alpha')$.
\end{itemize}
Let $\delta=\sup_{i\in \omega} (\max (a_i))=\sup_{i\in \omega}(\gamma_{t_i})$ and $a=\bigcup_{i\in \omega} a_i$. Repeat the argument in the proof of (1) to get $t'$ such that $\gamma_{t'}=\delta$, which is a lower bound for $\langle t_i: i\in \omega\rangle$. By the property that for any $i\in \omega$ and $\alpha<\alpha'\in \kappa$, $t_{i+1}(\gamma_{t_{i+1}},\alpha)\neq t_{i+1}(\gamma_{t_{i+1}},\alpha')$, we can further ensure that $t'(\delta, \beta)\neq a$ for any $\beta<\kappa$. Since $t'\Vdash ``a_i\sqsubset \dot{C}$'' for all $i\in \omega$, $t'\Vdash ``a\sqsubset \dot{C}$ and $\delta\in \lim (\dot{C})$'', 
it follows that $a=t'(\delta, \beta)$ for some $\beta<\kappa$, which is a contradiction.
\end{enumerate}

\end{proof}

Suppose $\kappa$ is a supercompact cardinal. Just like the first paragraph in the proof of Theorem \ref{NotSimul}, in $V^{\mathrm{Coll}(\omega_1, <\kappa)}$, we know that $\mathrm{RC}$ is indestructible under $\omega_2$-directed closed forcings.
 If we follow by forcing with $\mathbb{P}_{\square(\lambda, \kappa)}$ in $V^{\mathrm{Coll}(\omega_1, <\kappa)}$, then we will have $\mathrm{RC}$ along with $\square(\lambda, \omega_2)$ in the final model. 
\begin{remark}
As mentioned in the introduction, $\mathrm{RC}$ is known to refute $\square(\lambda,\omega)$ for regular $\lambda\geq \omega_2$ (\cite{MR3600760}). By Theorem 1.4.2 in \cite{Weiss} that $\square(\lambda,\omega_1)$ refutes the $(\lambda,\omega_2)$-strong tree property and by Theorem 3.1 in \cite{MR3600760} that $\mathrm{RC} + \neg \mathrm{CH}$ implies the $(\lambda,\omega_2)$-strong tree property for all $\lambda\geq \omega_2$, it follows that $\mathrm{RC} + \neg \mathrm{CH}$ implies the failure of $\neg \square(\lambda,\omega_1)$ for all regular $\lambda\geq \omega_2$. Notice that $\square(\omega_2, \omega_1)$ is a just a consequence of $\mathrm{CH}$.
\end{remark}

\begin{question}
Is $\mathrm{RC}+\mathrm{CH} + \square(\lambda, \omega_1)$ consistent for regular $\lambda>\omega_2$?
\end{question}

The following observation imposes some restriction on models of $\mathrm{RC}+\mathrm{CH}+\square(\lambda, \omega_1)$, if there are any.

\begin{definition}
We say a cardinal $\kappa$ is \emph{generically strongly compact} via a class of forcings $\mathcal{P}$ if: for any $\lambda\geq \kappa$, there exists $\mathbb{P}\in \mathcal{P}$ such that for any generic $G\subset \mathbb{P}$ over $V$, in $V[G]$, there exist an elementary embedding $j: V\to M$ such that $\mathrm{crit}(j)=\kappa$ and some $Y\in M$ such that $j''\lambda \subset  Y$ with $M\models |Y|<j(\kappa)$.
\end{definition}

\begin{observation}
If $\omega_2$ is generically strongly compact via the class of proper forcings, then $\square(\lambda,\omega_1)$ fails for all regular $\lambda > \omega_2$.
\end{observation}
\begin{proof}
Suppose otherwise for the sake of contradiction. Let $\bar{C}=\langle C_{\alpha,i}: i<\omega_1, \alpha\in \lim (\lambda)\rangle$ be a $\square(\lambda,\omega_1)$-sequence in $V$ for some regular $\lambda>\omega_2$. Let $\mathbb{P}$ be a proper forcing such that whenever $G\subset \mathbb{P}$ is generic, in $V[G]$, we can find an elementary embedding $j: V\to M$ such that $\mathrm{crit}(j)=\omega_2$ and some $Y\in M$ such that $j''\lambda\subset Y$ with $M\models |Y|<j(\omega_2)$. Let $\gamma=\sup (j''\lambda)$. We may assume $Y\subset \gamma$. In $M$, let $C\in j(\bar{C})(\gamma)$. So $C\subset \gamma$ is a club. Consider $A=j^{-1}((\lim (C)) \cap Y)$. $A$ is unbounded in $\lambda$ since $(\lim (C)) \cap Y \supset (\lim (C))\cap j''\lambda$ and the latter is an $\omega$-club. For each $\alpha\in A$, as $j(\alpha)\in (\lim (C))\cap Y$, by coherence there exists $i_\alpha<\omega_1$ such that $C\cap j(\alpha)= j(C_{\alpha,i_\alpha})$. Find $A'\subset A$ unbounded in $\lambda$ and $i<\omega_1$ such that for all $\alpha\in A'$, $i_\alpha=i$. Then $\bigcup_{\alpha\in A'} C_{\alpha, i}$ threads $\bar{C}$ in $V[G]$, since for any $\alpha<\beta\in A'$,  $j(C_{\beta, i})\cap j(\alpha)=(C\cap j(\beta))\cap j(\alpha)=C\cap j(\alpha)=j(C_{\alpha,i})$ so by elementarity $C_{\beta,i}\cap \alpha=C_{\alpha,i}$. However, by Corollary 2.22 in \cite{MR3730566}, no proper forcing can introduce a thread to $\bar{C}$. This is a contradiction.
\end{proof}

\section{Rado's Conjecture and polarized partition relations}\label{polarized}
In this section, we investigate the relationship between $\mathrm{RC}$ and certain polarized partition relations concerning $\omega_1$ and $\omega_2$.

Given an ideal $I$ on $X$, we can consider the following equivalence relation: $A\sim B$ iff $A\Delta B\in I$. Then $P(X)/I$ is the poset that consists of all $\sim$-equivalent classes, where $p\leq q$ iff there exists $A\in p, B\in q$ such that $A\subset_I B$, namely $A-B\in I$.

\begin{definition}
Let $I$ be an ideal on $\omega_1$. Then 
\begin{itemize}
\item $I$ is \emph{precipitous} if whenever $U\subset P(\omega_1)/I$ is a generic ultrafilter, then $\mathrm{Ult}(V, U)$ is well-founded.
\item $I$ is \emph{presaturated} if $I$ is precipitous and forcing with $P(\omega_1)/I$ preserves $\omega_2^V$ as a cardinal.
\end{itemize}

\end{definition}

Todor\v{c}evi\'{c} in \cite{MR1127033} showed that Chang's Conjecture (CC) implies the polarized partition relation $\begin{pmatrix}
\omega_2 \\
\omega_1
\end{pmatrix} \to \begin{pmatrix}
\omega \\
\omega
\end{pmatrix}^{1,1}_{\omega}$. As $\mathrm{CC}$ is independent of the existence of a saturated ideal on $\omega_1$ in general (see for example Proposition 8.52 in \cite{MR2768692}), it is a natural question whether the same polarized partition relation follows from the existence of an ideal on $\omega_1$ with certain saturation property. It turns out that a fairly weak saturation property suffices to get the partition relation.

\begin{lemma}\label{presaturated}
If there exists a presaturated ideal $I$ on $\omega_1$, then  $\begin{pmatrix}
\omega_2 \\
\omega_1
\end{pmatrix} \to \begin{pmatrix}
\omega \\
\omega
\end{pmatrix}^{1,1}_{\omega}$.
\end{lemma}

\begin{proof}
Let $f: \omega_2\times \omega_1\to \omega$ be the given coloring. Let $G$ be the generic ultrafilter for $P(\omega_1)/I$ over $V$. Let $j: V\to M\simeq Ult(V, G)$ be the associated elementary embedding defined in $V[G]$. We know by the assumption that $\mathrm{crit}(j)=\omega_1^V$, ${}^\omega
M\cap V[G]\subset M$ and $\omega_2^V=\omega_1^{V[G]}$ (see Proposition 4.8 in \cite{MR2768692} for more details). In $V[G]$, $A=_{def}j''\omega^V_2-\omega^V_1\subset M$ is uncountable. Therefore in $V[G]$ it is possible to find $k\in \omega$ and uncountable $A_0\subset A$ such that for all $j(\eta)\in A_0$, $j(f)(j(\eta), \omega_1^V)=k$. Pick $\beta_0\in \omega_2^V$ such that $j(\beta_0)=\min A_0$. For each $j(\eta)\in A_0$, in $M$, $\exists \xi<j(\omega_1^V) $ such that $ j(f)(j(\beta_0),\xi)=j(f)(j(\eta),\xi)=k$. By the elementarity of $j$, we know that in $V$, $\exists \alpha_\eta<\omega_1^V$, $f(\beta_0,\alpha_\eta)=f(\eta,\alpha_\eta)=k$. Since $\omega_1^V$ is countable in $V[G]$, by the Pigeon Hole Principle, we can find some $\gamma_0\in \omega^V_1$ and an uncountable $A_1\subset A_0$ such that for all $j(\eta)\in A_1$, $\alpha_\eta=\gamma_0$. 

Recursively, suppose for some $n\in \omega$, we have defined $\gamma_0<\cdots < \gamma_n \in \omega_1^V,  \beta_0<\cdots <\beta_n \in \omega_2^V$ and uncountable $A_{n+1}\subset \cdots \subset A_0$ such that $j(\beta_j)\in A_j$ for all $j\leq n$. Let $\beta_{n+1}$ be such that $j(\beta_{n+1})=\min A_{n+1}$. For each $j(\eta)\in A_{n+1}$, we know that $M\models ``\exists\alpha_{\eta}<j(\omega_1^V), \alpha_{\eta}>\gamma_n \ \& \  j(f)(j(\beta_j),\alpha_{\eta})=k=j(f)(j(\eta),\alpha_{\eta})$ for all $j\leq n+1$''. By the elementarity of $j$, it is true that for each $j(\eta)\in A_{n+1}$, we have $V\models ``\exists\alpha_{\eta}<\omega_1^V, \alpha_{\eta}>\gamma_n \ \& \  f(\beta_j,\alpha_{\eta})=k=f(\eta,\alpha_{\eta})$ for all $j\leq n+1$''. We can then define a map from $A_{n+1}$ to $\omega_1^V$ in $V[G]$, sending $j(\eta)$ to $\alpha_\eta$. Since $\omega_1^V$ is countable in $V[G]$, there exist $\gamma_{n+1}>\gamma_n$ and an uncountable $A_{n+2}\subset A_{n+1}$ such that for all $j(\eta)\in A_{n+2}$, $\alpha_\eta=\gamma_{n+1}$. Repeat the process above, we define two sequences $\langle \gamma_n\in \omega_1^V : n \in \omega\rangle$ and $\langle \beta_n\in \omega_2^V : n \in \omega\rangle$ in $V[G]$.

We have ensured that in $V[G]$, $j(f)\restriction \{j(\beta_n): n\in \omega\}\times\{\gamma_n: n\in \omega\}\equiv k$. Since ${}^\omega M\cap V[G]\subset M$, we know that $M \models ``$ there exist $C\in [j(\omega_2^V)]^\omega$ and $D\in [j(\omega_1^V)]^\omega$
such that $j(f)\restriction C\times D\equiv k$''. By the elementarity of $j$, we know that $V\models ``\exists C\in [\omega_2^V]^\omega, \exists D\in [\omega_1^V]^\omega$
such that $f\restriction C\times D\equiv k$''.

\end{proof}

\begin{remark}\label{finitecase}
Essentially the same proof as in Lemma \ref{presaturated} with straightforward modifications shows that the existence of a presaturated ideal on $\omega_1$ implies that $\begin{pmatrix}
\omega_2 \\
\omega_1
\end{pmatrix} \to \begin{pmatrix}
k \\
\omega_1
\end{pmatrix}^{1,1}_{\omega}$ for any $k\in \omega$.
\end{remark}

\begin{cor}
$\mathrm{RC}$ implies $\begin{pmatrix}
\omega_2 \\
\omega_1
\end{pmatrix} \to \begin{pmatrix}
\omega \\
\omega
\end{pmatrix}^{1,1}_{\omega}$ and 
$\begin{pmatrix}
\omega_2 \\
\omega_1
\end{pmatrix} \to \begin{pmatrix}
k \\
\omega_1
\end{pmatrix}^{1,1}_{\omega}$ for any $k\in \omega$.
\end{cor}

\begin{proof}
There are two ways of seeing this. On one hand, $\mathrm{RC}$ implies Chang's Conjecture, as shown by Todor\v{c}evi\'{c} in \cite{MR1261218}. Then we can finish by the remark preceding Lemma \ref{presaturated}. On the other hand, $\mathrm{RC}$ implies the non-stationary ideal on $\omega_1$ is presaturated, as shown by Feng in \cite{MR1683892}. Then apply Lemma \ref{presaturated} and Remark \ref{finitecase}.
\end{proof}

It is now natural to ask if we can prove anything stronger, in particular, the next natural partition relation to consider is 
$\begin{pmatrix}
\omega_2 \\
\omega_1
\end{pmatrix} \to \begin{pmatrix}
\omega \\
\omega_1
\end{pmatrix}^{1,1}_{\omega}$, which is a consequence of the existence of an ideal $I$ on $\omega_1$ that is $(\aleph_2,\aleph_2,\aleph_0)$-saturated, namely for any $\{X_\alpha\in I^+: \alpha<\omega_2\}$, there exists $A\in [\omega_2]^{\aleph_2}$ such that for any $B\in [A]^{\aleph_0}$, $\bigcap_{\alpha\in B} X_\alpha\in I^+$ (see Laver \cite{MR673792} for more details). It turns out that $\mathrm{RC}$ does not decide the truth of this statement.

We use and modify a little the idea of Prikry (\cite{MR0297577}) to add an $\omega_2$-sequence of $\omega_1$-partitions of $\omega_1$. In our final model, there exists a collection of sets $\langle A_{\alpha,\beta} \subset \omega_1: \alpha<\omega_2, \beta<\omega_1\rangle$
such that 
\begin{itemize}
\item for each $\alpha<\omega_2$, $\{A_{\alpha,\beta}: \beta\in \omega_1\}$ is a partition of $\omega_1$,
\item $|\omega_1-\bigcup_{n\in \omega} A_{\alpha_n, \xi_n}|\leq \aleph_0$ for any distinct $\langle \alpha_n\in \omega_2 : n\in \omega\rangle$, and not necessarily distinct $\langle \xi_n\in \omega_1 : n\in \omega\rangle$.
\end{itemize}

\begin{notation}
Given a countable function $S$ with its domain a subset of $\omega_2\times \omega_1$, let $S_0\in [\omega_2]^{\leq \omega}$ denote the projection of $\mathrm{dom}(S)$ to its first coordinate and $S_1\in [\omega_1]^{\leq \omega}$ denote the projection of $\mathrm{dom}(S)$ to its second coordinate.
\end{notation}

$\mathbb{P}_{\omega_2,\omega_1}$ consists of pairs $(S,\mathcal{A})$ where $S: \omega_2\times \omega_1\to \omega_1$ is a countable partial function such that $S_1\in \omega_1$ and $\mathcal{A}$ is a countable collection of countably infinite partial functions from $\omega_2$ to $\omega_1$ closed under co-finite restrictions, namely for each $f\in \mathcal{A}$, if $A=^*\mathrm{dom}(f), A\subset \mathrm{dom}(f)$, then $f\restriction A\in \mathcal{A}$.

We say $(S',\mathcal{A}')\leq (S,\mathcal{A})$ iff $S'\supset S$ and $\mathcal{A}'\supset \mathcal{A}$ and for all $\beta\in S'_1 -   S_1$ and for all $f\in \mathcal{A}$, there exists $\alpha\in \mathrm{dom}(f)$ such that $(\alpha, \beta)\in \mathrm{dom}(S')$ and $S'(\alpha,\beta)=f(\alpha)$.

\begin{remark}
In general, for a regular cardinal $\kappa$, we can define $\mathbb{P}_{\kappa,\omega_1}$ analogously by simply replacing $\omega_2$ with $\kappa$.
\end{remark}

It is easy to see that the ordering is transitive.

\begin{claim}\label{niceChainClosure}
$\mathbb{P}_{\omega_2,\omega_1}$ is $\aleph_2$-c.c. and countably closed assuming $\mathrm{CH}$.
\end{claim}

\begin{proof}
Countable closure is immediate. To see $\aleph_2$-c.c, given a collection of conditions $p_i=(S_i, \mathcal{A}_i)$ for $i<\omega_2$, we apply the $\Delta$-System Lemma (see Lemma III.6.15 in \cite{MR2905394} for a proof) to get $A\in [\omega_2]^{\aleph_2}$ such that

\begin{itemize}
\item $(S_i)_1$ is the same for all $i\in A$.
\item $\mathrm{dom}(S_i)$ forms a $\Delta$-system with root $r$ where $S_i\restriction r$ are all the same for $i\in A$.
\end{itemize}
Now any two conditions with indices $A$ are easily seen to be compatible.

\end{proof}

\begin{claim}\label{EnlargeByOne}
For any $\alpha\in \omega_2$ and any $\beta\in \omega_1$ and any $p=(S_p, \mathcal{A}_p)\in \mathbb{P}_{\omega_2, \omega_1}$, there exists $p'\leq p$ such that $(\alpha,\beta)\in \mathrm{dom}(S_{p'})$.
\end{claim}

\begin{proof}
If $(\alpha,\beta)\in \mathrm{dom}(S_p)$, then we take $p'$ to be $p$. Otherwise, if $\beta\in (S_p)_1$, then just add $(\alpha,\beta,0)$ to $S_p$. If $\beta \not\in (S_p)_1$, then for each $f_n\in \mathcal{A}_p$, find distinct $\alpha_n\in \mathrm{dom}(f_n)$ different from $\alpha$, then add $(\alpha,\beta', 0)$ and $(\alpha_n, \beta', f_n(\alpha_n))$ for $n\in \omega$ to $S_p$ for each $\beta'\leq \beta$ such that $\beta'\not\in (S_p)_1$.

\end{proof}

\begin{claim}
For any infinite countable partial function $f: \omega_2\to\omega_1$ and any $p=(S_p, \mathcal{A}_p)\in \mathbb{P}$, there exists $p'\leq p$ such that $f\in \mathcal{A}_{p'}$.
\end{claim}

\begin{proof}
Define $p'=(S_{p'}, \mathcal{A}_{p'})$ such that $S_{p'}=S_p$ and $\mathcal{A}_{p'}$ is the closure of $\mathcal{A}_p\cup \{f\}$ under co-finite restrictions. It is easy to verify that $p'$ is as desired.
\end{proof}

\begin{lemma}[Prikry, \cite{MR0297577}]\label{prikry}
Assume $V\models \mathrm{CH}$. Then in $V^{\mathbb{P}_{\omega_2,\omega_1}}$, $\begin{pmatrix}
\omega_2 \\
\omega_1
\end{pmatrix} \not\to \begin{bmatrix}
\omega \\
\omega_1
\end{bmatrix}^{1,1}_{\omega_1}$. Namely, there exists $f: \omega_2\times \omega_1\to \omega_1$ such that for any $A\in [\omega_2]^\omega$ and $B\in [\omega_1]^{\omega_1}$, $f'' A\times B = \omega_1$. This clearly implies $\begin{pmatrix}
\omega_2 \\
\omega_1
\end{pmatrix}\not \to \begin{pmatrix}
\omega \\
\omega_1
\end{pmatrix}^{1,1}_{\omega}$.
\end{lemma}

\begin{definition}\label{genericSPCP}
We say that $\omega_2$ is \emph{generically supercompact via countably closed forcing} if for any cardinal $\theta\geq \omega_2$ and any set $x\in H(\theta^+)$, there exists a countably closed forcing $\mathbb{M}$ such that whenever $H\subset \mathbb{M}$ is generic over $V$, in $V[H]$, there exists an elementary embedding $j: V\to M$ where $M$ is some transitive class such that 
\begin{enumerate}
\item $\mathrm{crit}(j)=\omega_2^V$,
\item $j''x\in M$,
\item $j(\omega_2^V)>\theta$.
\end{enumerate}

\end{definition}

\begin{remark}
In Definition \ref{genericSPCP}, instead of checking all $x\in H(\theta^+)$, we will get an equivalent definition even if we only check those $x's$ that are subsets of $\theta$. 
\end{remark}

\begin{lemma}\label{converge}
If $\omega_2$ is generically supercompact via countably closed forcing, then $\mathrm{RC}$ and $\mathrm{WRP}$ both hold.
\end{lemma}

\begin{proof}
To see that $\mathrm{RC}$ holds, fix some nonspecial tree $T$ of size $\theta\geq \omega_2$. We may assume $T$ is of the form $(\theta, <_T)$. Fix $\mathbb{M}$ that witnesses the generic supercompactness of $\omega_2$ with respect to the parameters $\theta$ and $T$. Let $H\subset \mathbb{M}$ be generic over $V$. Consider $Y=(j''\theta, <_{j(T)})\in M$. In $V[H]$, since $(T, <_T)\simeq (j''\theta,<_{j(T)})$ and $T$ remains nonspecial by Claim \ref{notspecializedclosed}, $Y$ is nonspecial in $V[H]$, hence in $M$. As $M\models `` |j'' \theta| = |\theta|\leq \theta < j(\omega_2^V)$'', by the elementarity of $j$, in $V$, it is true that there exists a nonspecial subtree of $T$ of size $\aleph_1$.

To see that $\mathrm{WRP}$ holds, fix some regular $\lambda\geq \omega_2$ and a stationary set $S\subset [\lambda]^\omega$. Let $\mathbb{M}$ be the witness to the generic supercompactness of $\omega_2$ with parameters $\lambda^\omega$ and $S$. Let $H\subset \mathbb{M}$ be generic over $V$ and $j: V\to M$ be the corresponding elementary embedding in $V[H]$. Consider $Y=j''S=\{j(x): x\in S\}=\{j''x: x\in S\}\in M$. Note $j''\lambda\in M$ since $j''\lambda = j''\bigcup S= \bigcup j''S$. It is easy to see that in $V[H]$, $S$ is stationary in $[\lambda]^\omega$ iff $j'' S$ is stationary in $[j''\lambda]^\omega$. Since $\mathbb{M}$ is countably closed hence proper, $S$ remains a stationary subset of $[\lambda]^\omega$ in $V[H]$. Hence in $V[H]$, $j''S$ is a stationary subset of $[j''\lambda]^\omega$. Since $j'' S, [j''\lambda]^\omega \in M$, we have that $M\models `` j'' S$ is a stationary subset of $[j''\lambda]^\omega$''. Also note that in $M$, we have that $|j''\lambda|=|\lambda|\leq \lambda \leq \lambda^\omega < j(\omega_2^V)$, $j'' S \subset j(S)$ and $\omega_1\subset j''\lambda$. Therefore, by the elementarity of $j$, it is true in $V$ that there exists $W\subset \lambda$ of size $\aleph_1$ containing $\omega_1$ such that $S\cap [W]^\omega$ is a stationary subset of $[W]^\omega$.

\end{proof}

\begin{lemma}\label{factor}
Let $\kappa$ be a supercompact cardinal.
In $V[G][H]$ where $G$ is generic for $\mathrm{Coll}(\omega_1,<\kappa)$ and $H$ is generic for $(\mathbb{P}_{\omega_2,\omega_1})^{V[G]}$ over $V[G]$, $\mathrm{WRP}$ and $\mathrm{RC}$ both hold.
\end{lemma}

\begin{proof}

We will show that in $V[G][H]$, $\kappa=\omega_2$ is generically supercompact via countably closed forcing. Then we can finish by Lemma \ref{converge}. 

First notice that $\mathbb{P}_{\omega_2,\omega_1}$ defined in $V[G]$ is the same as $\mathbb{P}_{\kappa,\omega_1}$ defined in $V$ since $\mathrm{Coll}(\omega_1, < \kappa)$ is countably closed. In $V$, fix $\theta\geq \kappa$ and a $(\mathrm{Coll}(\omega_1,<\kappa)\times \mathbb{P}_{\kappa,\omega_1})$-name $\dot{x}$ for a subset of $\theta$.
Let $\lambda>\theta$ be some sufficiently large regular cardinal, specifically larger than the cardinality of any nice $(\mathrm{Coll}(\omega_1,<\kappa)\times \mathbb{P}_{\kappa,\omega_1})$-name of a subset of $\theta$. Fix $j: V\to M$, an embedding witnessing $\kappa$ is $\lambda$-supercompact, namely $\mathrm{crit}(j)=\kappa$, $j(\kappa)>\lambda$ and ${}^\lambda M\subset M$. Observe that $j(\mathrm{Coll}(\omega_1,<\kappa)\times \mathbb{P}_{\kappa,\omega_1})=\mathrm{Coll}(\omega_1,<j(\kappa))\times \mathbb{P}_{j(\kappa),\omega_1}$. 

Let $R\subset \mathrm{Coll}(\omega_1, [\kappa,<j(\kappa))$ be generic over $V[G][H]$. We can lift $j$ to $j: V[G]\to M[G][R]$ (we slightly abuse the notation by using the same $j$). Notice in $V[G]$, $j\restriction \mathbb{P}_{\omega_2,\omega_1}: \mathbb{P}_{\omega_2,\omega_1}\to j(\mathbb{P}_{\omega_2,\omega_1})$ is a complete embedding since $V[G]\models \mathrm{CH}$ so by Claim \ref{niceChainClosure}, $\mathbb{P}_{\omega_2,\omega_1}$ is $\aleph_2$-c.c. in $V[G]$. Therefore, as $\mathrm{crit}(j)=\kappa=(\omega_2)^{V[G]}$, each maximal antichain of $\mathbb{P}_{\omega_2,\omega_1}$ in $V[G]$ is mapped to its pointwise image, which is a maximal antichain of $j(\mathbb{P}_{\omega_2,\omega_1})$.
Notice also that $j\restriction \mathbb{P}_{\omega_2, \omega_1}$ is just the identity function on $\mathbb{P}_{\omega_2, \omega_1}$.

\begin{claim}\label{quotientCountablyClosed}
$j(\mathbb{P}_{\omega_2,\omega_1})/H$ is countably closed in $V[G][R][H]$.
\end{claim}

\begin{proof}[Proof of the Claim]
Let $\mathbb{P}=(\mathbb{P}_{\kappa,\omega_1})^V=(\mathbb{P}_{\omega_2,\omega_1})^{V[G]}$.
Fix a decreasing sequence $\langle p_n \in j(\mathbb{P}_{\omega_2,\omega_1})/H : n\in \omega\rangle$. We claim that $q=\bigcup_n p_n \in j(\mathbb{P})/H$, i.e. we show for all $h\in H$, $q$ is compatible with $h$. This clearly implies that $q$ is a lower bound for $\langle p_n: n\in \omega\rangle$ in $j(\mathbb{P}_{\omega_2,\omega_1})/H$.

Note that for each $n\in \omega$, $S_{q_n}$ must agree with $S_H=\bigcup\{S_t: t\in H\}$ which is a total function on $\kappa\times \omega_1\to \omega_1$. 
Fix $h\in H$. By the observation above, we know $S_h$ is consistent with $S_q$. By extending $h$ and by Claim \ref{EnlargeByOne}, we may assume $(S_h)_1> (S_q)_1$.

Fix $\beta \in (S_h)_1-(S_q)_1$. For each $f\in \mathcal{A}_q$, there exists $n\in \omega$ such that $f\in \mathcal{A}_{p_n}$. By the fact that $h$ and $p_n$ are compatible, there exists $\alpha\in \mathrm{dom}(f)$ such that $(\alpha,\beta, f(\alpha))$ can be added to $S_h$. Since $\mathcal{A}_{p_n}$ is closed under co-finite restrictions, we know in fact there are infinitely many such candidates. What is left to do is to recursively build a condition $h'$ extending both $h$ and $q$ such that $(S_{h'})_1 = (S_h)_1$ and $\mathcal{A}_{h'}=\mathcal{A}_h\cup \mathcal{A}_q$. 

Enumerate $(S_h)_1-(S_q)_1$ as $\{\beta_i: i\in \omega\}$. Define $\langle l_i: i\in \omega\rangle$ where each $l_i$ is a countable function from $j(\kappa)\times \omega_1\to \omega_1$. For each $i\in \omega$, for each $f\in \mathcal{A}_q$, find $\alpha=\alpha_f\in \mathrm{dom}(f)$ such that $\alpha_f\neq \alpha_g$ for any $f\neq g \in \mathcal{A}_q$ and $l_i=\{(\alpha_f, \beta_i, f(\alpha_f)): f\in \mathcal{A}_q\}$ is compatible with $S_h$. Also $l_i$ is necessarily compatible with $S_q$ since $\beta_i\not\in (S_q)_1$. Notice for $i\neq j\in \omega$, $l_i$ is compatible with $l_j$. Let $l_\omega=\bigcup_{i\in \omega}l_i \cup S_h \cup S_q$.
It is easy to verify that $h'=(l_\omega, \mathcal{A}_q\cup \mathcal{A}_h)$ works.

\end{proof}

Let $L$ be generic for $j(\mathbb{P}_{\omega_2,\omega_1})/H$ over $V[G][R][H]$. We can further lift $j$ to $j: V[G][H]\to M[G][R][H][L]$. By the choice of $\lambda$, $x=(\dot{x})^{G\times H}\in M[G][H]$. Since $j''\theta\in M$, we have that $j'' x \in M$. Since $\mathrm{Coll}(\omega_1, [\kappa, <j(\kappa))\times j(\mathbb{P}_{\omega_2, \omega_1})/H$ is countably closed in $V[G][H]$ by Claim \ref{quotientCountablyClosed}, we have found a countably closed forcing that witnesses $\kappa$ (recall that ${V[G][H]}\models \kappa=\omega_2$) is generically supercompact with respect to the parameters $\theta$ and $x$.

\end{proof}

Combining Lemma \ref{prikry} and Lemma \ref{factor} we have the following.

\begin{theorem}
Let $\kappa$ be a supercompact cardinal. Then there exists a forcing extension in which $\kappa=\omega_2$, $\mathrm{RC}$ and $\mathrm{WRP}$ both hold, and $\begin{pmatrix}
\omega_2 \\
\omega_1
\end{pmatrix}\not \to \begin{pmatrix}
\omega \\
\omega_1
\end{pmatrix}^{1,1}_{\omega}$.
\end{theorem}

\section{Acknowledgment}
 I thank James Cummings and Stevo Todor\v{c}evi\'{c} for many helpful discussions and comments. I am especially grateful to the anonymous referee for many extremely helpful comments and suggestions which lead to significant improvements of this paper.

\bibliographystyle{plain}
\bibliography{bib}

\Addresses

\end{document}